\documentclass[a4paper,12pt,final]{amsart}
\usepackage{times,a4wide,mathrsfs,amssymb,amsmath,amsthm}

\newcommand{\C}{\mathbb{C}}

\newcommand{\LLL}{\mathbb{L}}
\newcommand{\QQ}{\mathbb{Q}}
\newcommand{\NN}{\mathbb{N}}
\newcommand{\PP}{\mathbb{P}}

\newcommand{\OO}{\mathcal O}
\newcommand{\Ss}{\mathcal S}
\newcommand{\TT}{\mathcal T}

\newcommand{\Sy}{\mathfrak S}

\newcommand{\DD}{\mathcal D}
\newcommand{\XX}{\mathcal X}

\newcommand{\VV}{\mathcal V}
\newcommand{\WW}{\mathcal W}
\newcommand{\Cc}{\mathcal C}

\newcommand{\EE}{\mathcal E}
\newcommand{\MM}{\mathcal M}

\newcommand{\gr}{\hbox{Gr}}
\newcommand{\wt}{\widetilde}
\newcommand{\rom}{\romannumeral}

\DeclareMathOperator{\aut}{Aut}

\DeclareMathOperator{\ide}{id}

\newtheorem{theorem}{Theorem}[section]
\newtheorem{claim}[theorem]{Claim}
\newtheorem{lemma}[theorem]{Lemma}
\newtheorem{sublemma}[theorem]{Sublemma}
\newtheorem{corollary}[theorem]{Corollary}
\newtheorem{proposition}[theorem]{Proposition}
\newtheorem{conjecture}[theorem]{Conjecture}
\newtheorem{remark}[theorem]{Remark}
\newtheorem{definition}[theorem]{Definition}
\newtheorem{convention}{Conventions}

\newtheorem{notation}[theorem]{Notation}

\newtheorem{nonumbering}{Theorem}

\newtheorem{nonumberingc}{Corollary}

\newtheorem{nonumberingt}{Acknowledgements}

\begin{document}
\author[Robert Laterveer]
{Robert Laterveer}

\address{Institut de Recherche Math\'ematique Avanc\'ee,
CNRS -- Universit\'e 
de Strasbourg,\
7 Rue Ren\'e Des\-car\-tes, 67084 Strasbourg CEDEX,
FRANCE.}
\email{robert.laterveer@math.unistra.fr}

\title{Algebraic cycles and EPW cubes}

\begin{abstract} Let $X$ be a hyperk\"ahler variety with an anti--symplectic involution $\iota$. According to Beauville's conjectural ``splitting property'', the Chow groups of $X$ should split in a finite number of pieces such that the Chow ring has a bigrading. The Bloch--Beilinson conjectures predict how $\iota$ should act on certain of these pieces of the Chow groups. We verify part of this conjecture for a $19$--dimensional family of hyperk\"ahler sixfolds that are ``double EPW cubes'' (in the sense of Iliev--Kapustka--Kapustka--Ranestad). This has interesting consequences for the Chow ring of the quotient $X/\iota$, which is an ``EPW cube'' (in the sense of Iliev--Kapustka--Kapustka--Ranestad).
\end{abstract}


\maketitle

\section{Introduction}

For a smooth projective variety $X$ over $\C$, let us write
  \[ A^i(X):=CH^i(X)\otimes{\QQ}\] 
 to denote the Chow groups of $X$ (i.e. codimension $i$ algebraic cycles on $X$ modulo rational equivalence), with $\QQ$--coefficients. As is well--known (and explained for instance in \cite{J2}, \cite{Vo}, \cite{MNP}), the Bloch--Beilinson conjectures form a powerful and coherent heuristic guide, useful in formulating concrete predictions about Chow groups and their relation to cohomology. 
This note is about one instance of such a prediction, concerning non--symplectic involutions on hyperk\"ahler varieties.

Let $X$ be a hyperk\"ahler variety (i.e., a projective irreducible holomorphic symplectic manifold, cf. \cite{Beau0}, \cite{Beau1}), and suppose $X$ has an anti--symplectic involution 
$\iota$. 
The action of $\iota$ on the subring $H^{\ast,0}(X)$ is well--understood: we have
 \[ \begin{split} \iota^\ast=  -\ide \colon\ \ \ &H^{2i,0}(X)\ \to\ H^{2i,0}(X)\ \ \ \hbox{for}\ i\ \hbox{odd}\ ,\\
                        \iota^\ast=  \ide \colon\ \ \ &H^{2i,0}(X)\ \to\ H^{2i,0}(X)\ \ \ \hbox{for}\ i\ \hbox{even}\ .\\
                  \end{split}\]
                  
The action of $\iota$ on the Chow ring $A^\ast(X)$ is more mysterious.                      
To state the conjectural behaviour,
we will now assume the Chow ring of $X$ has a bigraded ring structure $A^\ast_{(\ast)}(X)$, where each $A^i(X)$ splits into pieces
  \[ A^i(X) =\bigoplus_j A^i_{(j)}(X)\ ,\]
and the piece $A^i_{(j)}(X)$ is isomorphic to the graded $\gr^j_F A^i(X)$ for the Bloch--Beilinson filtration that conjecturally exists for all smooth projective varieties.   
 (Such a bigrading $A^\ast_{(\ast)}(-)$ is expected to exist for all hyperk\"ahler varieties; this is Beauville's conjectural ``splitting property'' \cite{Beau3}.) 
 
 Since the pieces $A^i_{(i)}(X)$ and $A^{\dim X}_i(X)$ should only depend on the subring $H^{\ast,0}(X)$, we are led to the following conjecture:
 
 \begin{conjecture}\label{conj} Let $X$ be a hyperk\"ahler variety of dimension $2m$, and let $\iota\in\aut(X)$ be an anti--symplectic involution. Then
   \[  \begin{split}  \iota^\ast= (-1)^i \ide\colon\ \ \ &A^{2i}_{(2i)}(X)\ \to\ A^{2i}(X)\ ,\\
                           \iota^\ast= (-1)^i \ide\colon\ \ \ &A^{2m}_{(2i)}(X)\ \to\ A^{2m}(X)\ .\\
                    \end{split}\]       
  \end{conjecture}
 
 This conjecture is studied, and proven in some particular cases, in \cite{EPW}, \cite{HKnonsymp}, \cite{BlochHK4}, \cite{ChowEPW}, \cite{ChowHK4nonsympinv}.
 The aim of this note is to provide some more examples where conjecture \ref{conj} is verified, by considering ``double EPW cubes'' in the sense of \cite{IKKR} (cf. also subsection \ref{sscube} below). 
 A double EPW cube is a $6$--dimensional hyperk\"ahler variety $X_A$, constructed as double cover
   \[ X_A\ \to\ D_2^A\ ,\]
   where $D_2^A$ is a slightly singular subvariety of a Grassmannian (the variety $D_2^A$ is called an ``EPW cube''). As shown in \cite{IKKR}, double EPW cubes correspond to a $20$--dimensional irreducible (and unirational) component of the moduli space of hyperk\"ahler sixfolds.
   A double EPW cube $X_A$ comes equipped with the covering involution
   \[ \iota_A\colon\ \ \ X_A\ \to\ X_A \]
   which is anti--symplectic (remark \ref{anti}).
 
  The main result of this note is a partial verification of conjecture \ref{conj} for a $19$--dimensional family of double EPW cubes:

\begin{nonumbering}[=theorem \ref{main}] Let $X$ be a double EPW cube, and assume $X=X_A$ for $A\in \Delta^1$ general (where $\Delta^1\subset LG^1_\nu$ is the divisor of theorem \ref{ikkr}). Let $\iota=\iota_A\in\aut(X)$ be the anti--symplectic involution. Then
  \[   \begin{split}  \iota^\ast&=-\ide\colon\ \ \ A^6_{(2)}(X)\ \to\ A^6(X)\ ,\\
                        (\Pi_2^X)_\ast \iota^\ast&=-\ide\colon\ \ \ A^2_{(2)}(X)\ \to\ A^2_{(2)}(X)\ .
                        \end{split}\]      
    \end{nonumbering}
        
The 
divisor $\Delta^1$ is such that for $A\in\Delta^1$ general, the double EPW cube $X_A$ is birational to a Hilbert scheme $(S_A)^{[3]}$, where $S_A$ is a degree $10$ $K3$ surface. Since Hilbert schemes $S^{[m]}$ of $K3$ surfaces $S$ have a multiplicative Chow--K\"unneth decomposition \cite{V6}, double EPW cubes $X=X_A$ as in theorem \ref{main} have a bigraded Chow ring $A^\ast_{(\ast)}(X)$ (cf. corollary \ref{mck} below). The correspondence $\Pi_2^X$ is a projector on $A^2_{(2)}(X)$.

 To prove theorem \ref{main}, we employ the method of ``spread'' of algebraic cycles as developed by Voisin \cite{V0}, \cite{V1}. 
%
%
         Theorem \ref{main} has some rather striking consequences for the Chow ring of the EPW cubes in the $19$--dimensional family under consideration (these consequences exploit the existence of a multiplicative Chow--K\"unneth decomposition for $X$ as in theorem \ref{main}):
  
  \begin{nonumberingc}[=corollary \ref{cor}]  Let $D=D_2^A$ be an EPW cube for $A\in\Delta^1$ general. 
  
  \noindent
  (\rom1) Let $a\in A^6(D)$ be a $0$--cycle which is either in the image of the intersection product map
    \[ A^2(D)\otimes A^2(D)\otimes A^2(D)\ \to\ A^6(D)\ ,\] 
    or in the image of the intersection product map
    \[ A^3(D)\otimes A^2(D)\otimes A^1(D)\ \to\ A^6(D)\ .\]
    Then $a$ is rationally trivial if and only if $a$ has degree $0$.
    
   \noindent
   (\rom2) Let $a\in A^5(D)$ be a $1$--cycle which is in the image of the intersection product map
    \[ A^2(D)\otimes A^2(D)\otimes A^1(D)\ \to\ A^5(D)\ .\]
    Then $a$ is rationally trivial if and only if $a$ is homologically trivial. 
  \end{nonumberingc}  
  
  (NB: the EPW cube $D$ is not smooth, but it is a quotient of a smooth variety; as such, the Chow groups of $D$ still have a ring structure, cf. subsection \ref{ssquot} below.)
  
  Corollary \ref{cor} is similar to multiplicative results in the Chow ring of $K3$ surfaces \cite{BV}, in the Chow ring of Hilbert schemes of $K3$ surfaces and of abelian surfaces \cite{V6}, and in the Chow ring of Calabi--Yau complete intersections \cite{V13}, \cite{LFu}. A more general version of corollary \ref{cor}, concerning certain product varieties, can be proven similarly (corollary \ref{cor2}).
  
  It is my hope this note will stimulate further research on this topic. For one thing, it would be interesting to prove theorem \ref{main} for {\em all\/} double EPW cubes, and corollary \ref{cor} for {\em all\/} EPW cubes.

 \vskip0.6cm

\begin{convention} In this article, the word {\sl variety\/} will refer to a reduced irreducible scheme of finite type over $\C$. A {\sl subvariety\/} is a (possibly reducible) reduced subscheme which is equidimensional. 

{\bf All Chow groups will be with rational coefficients}: we will denote by 
  \[ A_j(X):=CH^j(X)\otimes\QQ\] 
  the Chow group of $j$--dimensional cycles on $X$ with $\QQ$--coefficients. For $X$ smooth of dimension $n$ we will write
    \[ A^i(X):=A_{n-i}(X)\ .\]
  The notations $A^i_{hom}(X)$, $A^i_{AJ}(X)$, $A^i_{alg}(X)$ will be used to indicate the subgroups of homologically trivial, resp. Abel--Jacobi trivial, resp. algebraically trivial cycles.
For a morphism $f\colon X\to Y$, we will write 
  \[\Gamma_f\in A_\ast(X\times Y)\] 
  for the graph of $f$.
The contravariant category of Chow motives (i.e., pure motives with respect to rational equivalence as in \cite{Sc}, \cite{MNP}) will be denoted $\MM_{\rm rat}$.



We will use $H^j(X)$ 
to indicate singular cohomology $H^j(X,\QQ)$.

\end{convention}

\section{Preliminaries}

\subsection{Quotient varieties}
\label{ssquot}

\begin{definition} A {\em projective quotient variety\/} is a variety
  \[ X=Y/G\ ,\]
  where $Y$ is a smooth projective variety and $G\subset\hbox{Aut}(Y)$ is a finite group.
  \end{definition}
  
 \begin{proposition}[Fulton \cite{F}]\label{quot} Let $X$ be a projective quotient variety of dimension $n$. Let $A^\ast(X)$ denote the operational Chow cohomology ring. The natural map
   \[ A^i(X)\ \to\ A_{n-i}(X) \]
   is an isomorphism for all $i$.
   \end{proposition}
   
   \begin{proof} This is \cite[Example 17.4.10]{F}.
      \end{proof}

\begin{remark} It follows from proposition \ref{quot} that the formalism of correspondences goes through unchanged for projective quotient varieties (this is also noted in \cite[Example 16.1.13]{F}). We can thus consider motives $(X,p,0)\in\MM_{\rm rat}$, where $X$ is a projective quotient variety and $p\in A^n(X\times X)$ is a projector. For a projective quotient variety $X=Y/G$, one readily proves (using Manin's identity principle) that there is an isomorphism
  \[  h(X)\cong h(Y)^G:=(Y,\Delta^G_Y,0)\ \ \ \hbox{in}\ \MM_{\rm rat}\ ,\]
  where $\Delta^G_Y$ denotes the idempotent ${1\over \vert G\vert}{\sum_{g\in G}}\Gamma_g$.  
  \end{remark}

\subsection{MCK decomposition}
\label{ss1}

\begin{definition}[Murre \cite{Mur}] Let $X$ be a smooth projective variety of dimension $n$. We say that $X$ has a {\em CK decomposition\/} if there exists a decomposition of the diagonal
   \[ \Delta_X= \pi_0+ \pi_1+\cdots +\pi_{2n}\ \ \ \hbox{in}\ A^n(X\times X)\ ,\]
  such that the $\pi_i$ are mutually orthogonal idempotents in $A^n(X\times X)$ and $(\pi_i)_\ast H^\ast(X)= H^i(X)$.
  
  (NB: ``CK decomposition'' is shorthand for ``Chow--K\"unneth decomposition''.)
\end{definition}

\begin{remark} The existence of a CK decomposition for any smooth projective variety is part of Murre's conjectures \cite{Mur}, \cite{J2}, \cite{J4}. 
\end{remark}

\begin{definition}[Shen--Vial \cite{SV}] Let $X$ be a smooth projective variety of dimension $n$. Let $\Delta_X^{sm}\in A^{2n}(X\times X\times X)$ be the class of the small diagonal
  \[ \Delta_X^{sm}:=\bigl\{ (x,x,x)\ \vert\ x\in X\bigr\}\ \subset\ X\times X\times X\ .\]
  An {\em MCK decomposition\/} is a CK decomposition $\{\pi^X_i\}$ of $X$ that is {\em multiplicative\/}, i.e. it satisfies
  \[ \pi^X_k\circ \Delta_X^{sm}\circ (\pi^X_i\times \pi^X_j)=0\ \ \ \hbox{in}\ A^{2n}(X\times X\times X)\ \ \ \hbox{for\ all\ }i+j\not=k\ .\]
  
 (NB: ``MCK decomposition'' is shorthand for ``multiplicative Chow--K\"unneth decomposition''.) 
  
 A {\em weak MCK decomposition\/} is a CK decomposition $\{\pi^X_i\}$ of $X$ that satisfies
    \[ \Bigl(\pi^X_k\circ \Delta_X^{sm}\circ (\pi^X_i\times \pi^X_j)\Bigr){}_\ast (a\times b)=0 \ \ \ \hbox{for\ all\ } a,b\in\ A^\ast(X)\ .\]
  \end{definition}
  
  \begin{remark} The small diagonal (seen as a correspondence from $X\times X$ to $X$) induces the {\em multiplication morphism\/}
    \[ \Delta_X^{sm}\colon\ \  h(X)\otimes h(X)\ \to\ h(X)\ \ \ \hbox{in}\ \MM_{\rm rat}\ .\]
 Suppose $X$ has a CK decomposition
  \[ h(X)=\bigoplus_{i=0}^{2n} h^i(X)\ \ \ \hbox{in}\ \MM_{\rm rat}\ .\]
  By definition, this decomposition is multiplicative if for any $i,j$ the composition
  \[ h^i(X)\otimes h^j(X)\ \to\ h(X)\otimes h(X)\ \xrightarrow{\Delta_X^{sm}}\ h(X)\ \ \ \hbox{in}\ \MM_{\rm rat}\]
  factors through $h^{i+j}(X)$.
  
  If $X$ has a weak MCK decomposition, then setting
    \[ A^i_{(j)}(X):= (\pi^X_{2i-j})_\ast A^i(X) \ ,\]
    one obtains a bigraded ring structure on the Chow ring: that is, the intersection product sends $A^i_{(j)}(X)\otimes A^{i^\prime}_{(j^\prime)}(X) $ to  $A^{i+i^\prime}_{(j+j^\prime)}(X)$.
    
      It is expected (but not proven !) that for any $X$ with a weak MCK decomposition, one has
    \[ A^i_{(j)}(X)\stackrel{??}{=}0\ \ \ \hbox{for}\ j<0\ ,\ \ \ A^i_{(0)}(X)\cap A^i_{hom}(X)\stackrel{??}{=}0\ ;\]
    this is related to Murre's conjectures B and D, that have been formulated for any CK decomposition \cite{Mur}.

  The property of having an MCK decomposition is severely restrictive, and is closely related to Beauville's ``(weak) splitting property'' \cite{Beau3}. For more ample discussion, and examples of varieties with an MCK decomposition, we refer to \cite[Section 8]{SV}, as well as \cite{V6}, \cite{SV2}, \cite{FTV}.
    \end{remark}
    
  \begin{lemma}\label{hk} Let $X, X^\prime$ be birational hyperk\"ahler varieties. Then $X$ has an MCK decomposition if and only if $X^\prime$ has one.
  \end{lemma}
  
  \begin{proof} This is noted in \cite[Introduction]{V6}; the idea is that Rie\ss's result \cite{Rie} implies that $X$ and $X^\prime$ have isomorphic Chow motives and the isomorphism is compatible with the multiplicative structure. (For a detailed proof, cf. \cite[Lemma 2.13]{EPW}.)
  \end{proof}

\subsection{MCK for $S^{[m]}$}


\begin{theorem}[Vial \cite{V6}]\label{charles} Let $S$ be a projective $K3$ surface, and let $X=S^{[m]}$ be the Hilbert scheme of length $m$ subschemes of $S$. Then $X$ has a self--dual MCK decomposition $\{ \Pi^X_i\}$. In particular, $A^\ast(X)=A^\ast_{(\ast)}(X)$ is a bigraded ring, where
  \[ A^i(X)=\bigoplus_{j= 2i-2n }^i A^i_{(j)}(X)\ ,\]
  and $A^i_{(j)}(X)=0$ for $j$ odd.
\end{theorem}

\begin{proof} This is \cite[Theorems 1 and 2]{V6}.
\end{proof}

\begin{remark} Let $X$ be as in theorem \ref{charles} and suppose $m=2$ (i.e. $X=S^{[2]}$ is a hyperk\"ahler fourfold). Then the bigrading $A^\ast_{(\ast)}(X)$ of theorem \ref{charles} has an interesting alternative description in terms of a Fourier operator on Chow groups \cite{SV}. For $m>2$, there is no such ``Fourier operator'' description of the bigrading $A^\ast_{(\ast)}(S^{[m]})$; the bigrading is defined exclusively by an MCK decomposition.

Another point particular to $m=2$ is that (thanks to \cite{SV}) we know that
  \[ A^i_{(j)}(S^{[2]})=0\ \ \ \forall j<0\ .\]
  This vanishing statement is (conjecturally true but) open for $S^{[m]}$ with $m>2$.
\end{remark}

Any $K3$ surface $S$ has an MCK decomposition \cite[Example 8.17]{SV}. Since this property is stable under products \cite[Theorem 8.6]{SV}, $S^m$ also has an MCK decomposition. The following lemma records a basic compatibility between the bigradings on $A^\ast(S^{[m]})$ and on $A^\ast(S^m)$:

\begin{lemma}\label{compat} Let $S$ be a $K3$ surface, and let $X=S^{[m]}$. Let $\Phi\in A^{2m}(X\times S^m)$ be the correspondence coming from the diagram
  \[ \begin{array}[c]{ccc}
          S^{[m]} & \xleftarrow{} &\wt{S^m}\\
         {\scriptstyle h} \downarrow\ \ \  && \downarrow\\
          S^{(m)} & \xleftarrow{g}& S^m\\
          \end{array}\]
       (the arrow labelled $h$ is the Hilbert--Chow morphism; the right vertical arrow is the blow--up of the diagonal). Then
     \[  \begin{split} &(\Phi)_\ast R(X)\ \subset\ R(S^m)\ ,\\     
                            &({}^t\Phi)_\ast R(S^m)\ \subset\ R(X)\ ,\\
                            \end{split}\]
                    where $R()=A^{2m}_{(j)}()$ or $A^2_{(2)}()$.        
\end{lemma}

\begin{proof} We first prove the statement for ${}^t \Phi$. By construction of the MCK decomposition for $X$, there is a relation
  \begin{equation}\label{XS}  \Pi_k^X= {1\over m}\ {}^t \Phi\circ \Pi_k^{S^m}\circ \Phi + \hbox{Rest}\ \ \ \hbox{in}\ A^{2m}(X\times X)\ , \ \ \ (k=0,2,4,\ldots,4m)\ ,\end{equation}
  where $\{\Pi_k^{S^m}\}$ is a product MCK decomposition for $S^m$, and  ``Rest'' is a term coming from various partial diagonals. For dimension reasons, the term ``Rest''  does not act on $A^{2m}(X)$ and on $A^2_{AJ}(X)$.
Since ${1\over m}\ {}^t \Phi\circ \Phi$ is the identity on $A^{2m}(X)$ and on $A^2_{hom}(X)=A^2_{AJ}(X)$, we can write
  \[ ({}^t \Phi)_\ast (\Pi_k^{S^m})_\ast = ({}^t \Phi \circ\Pi_k^{S^m})_\ast = ({1\over m}\ {}^t \Phi \circ \Phi \circ {}^t \Phi\circ\Pi_k^{S^m})_\ast\colon\ \ \ T(S^m)\ 
  \to\ T(X)\ ,\]
  where $T()$ is either $A^{2m}()$ or $A^2_{hom}()$.
  In view of sublemma \ref{sub} below, this implies
   \[ ({}^t \Phi)_\ast (\Pi_k^{S^m})_\ast =  ({1\over m}\ {}^t \Phi \circ \Pi_k^{S^m} \circ \Phi \circ {}^t \Phi)_\ast\colon\ \ \ T(S^m)\ \to\ T(X)\ .\]  
   But then, plugging in relation (\ref{XS}), we find 
   \[ ({}^t \Phi)_\ast (\Pi_k^{S^m})_\ast T(S^m)\ \subset\ (\Pi_k^X)_\ast T(X)\ .\]
  Taking $k=2$ and $T=A^2_{hom}()$, this proves
    \[ ({}^t \Phi)_\ast  A^2_{(2)}(S^m)\ \subset\  A^2_{(2)}(X)\ .\]
    Taking $k=4m-j$ and $T=A^{2m}()$, this proves
    \[ ({}^t \Phi)_\ast  A^{2m}_{(j)}(S^m)\ \subset\ A^{2m}_{(j)}(X)\ .\]
    
    The proof of the first statement of lemma \ref{compat} is similar: equality (\ref{XS}) implies that
    \[ \Phi_\ast (\Pi^X_k)_\ast ={1\over m}\bigl(  \Phi\circ  {}^t \Phi \circ \Pi_k^{S^m}\circ \Phi \bigr){}_\ast\colon\ \ \ T(X)\ \to\ T(S^m)\ .\]
    Using sublemma \ref{sub}, this slinks down to
    \[  \begin{split}    \Phi_\ast (\Pi^X_k)_\ast &={1\over m}\bigl(     \Pi_k^{S^m}\circ  \Phi\circ  {}^t \Phi\circ \Phi\bigr){}_\ast\\
                                                                    &= (\Pi_k^{S^m}\circ \Phi)_\ast\ \colon\ \ \ \ \ \ \ T(X)\ \to\ T(S^m)\ .\\
                                                                 \end{split}\]
                                                            This proves the first statement of lemma \ref{compat}.

   \begin{sublemma}\label{sub} There is commutativity
     \[ \bigl(\Phi\circ {}^t \Phi\circ \Pi_k^{S^m}\bigr){}_\ast =    \bigl(\Pi_k^{S^m}\circ\Phi\circ {}^t \Phi\bigr){}_\ast\ \ \ A^{i}(S^m)\ \to\  A^i(S^m)\ \ \ \forall i\ ,\ \ \ \forall k\ .\]
     \end{sublemma}
     
   To prove the sublemma, we remark that $h_\ast h^\ast=m\ide\colon A^i(S^{(m)})\to A^i(S^{(m)})$, and so
    \begin{equation}\label{seq} (\Phi\circ {}^t \Phi)_\ast = m\ g^\ast g_\ast = m (\sum_{\sigma\in\Sy_m} \Gamma_\sigma )_\ast\colon\ \ \   A^i(S^{m})\to A^i(S^{m})\ ,\end{equation}
    where the symmetric group $\Sy_m$ acts in the natural way on the product $S^m$. But $\{\Pi_k^{S^m}\}$, being a product decomposition, is symmetric and hence
    \[ \Gamma_\sigma\circ \Pi_k^{S^m} \circ \Gamma_{\sigma^{-1}}= (\sigma\times\sigma)^\ast \Pi_k^{S^m} =\Pi_k^{S^m}\ \ \ \hbox{in}\ A^{2m}(S^m\times S^m)\ \ \ \forall \sigma\in\Sy_m\ ,\ \ \ \forall k\ .\]
    This implies commutativity
    \[ \Gamma_\sigma\circ \Pi_k^{S^m} =  \Pi_k^{S^m}\circ \Gamma_\sigma\ \ \ \hbox{in}\ A^{2m}(S^m\times S^m)\ \ \ \forall \sigma\in\Sy_m\ ,\ \ \ \forall k\ .\]
  Combining with equation (\ref{seq}), this proves the sublemma.
     \end{proof}

 \begin{remark} Lemma \ref{compat} is probably true for any $(i,j)$ (i.e., the correspondence $\Phi$ should be ``of pure grade $0$'' in the language of \cite[Definition 1.1]{SV2}). I have not been able to prove this.
  \end{remark}

\subsection{Relative MCK for $S^m$}

\begin{notation} Let $\Ss\to B$ be a family (i.e., a smooth projective morphism). For $r\in\NN$, we write $\Ss^{r/B}$ for the relative $r$--fold fibre product
  \[ \Ss^{r/B}:= \Ss\times_B \Ss\times_B \cdots \times_B \Ss \ \]
  ($r$ copies of $\Ss$).
  \end{notation}

\begin{proposition}\label{prod} Let $\Ss\to B$ be a family of $K3$ surfaces. There exist relative correspondences 
   \[  \Pi_j^{\Ss^{m/B}}\ \ \in A^{2m}(\Ss^{m/B}\times \Ss^{m/B})\ \ \  (j=0,2,4,\ldots, 4m)\ ,\]
  such that for each $b\in B$, the restriction
  \[ \Pi_j^{(S_b)^m} := \Pi_j^{\Ss^{m/B}}\vert_{(S_b)^{2m}}\ \ \ \in A^4((S_b)^m\times (S_b)^m)\]
  defines a self--dual MCK decomposition for $(S_b)^m$.
  \end{proposition}

\begin{proof}

 On any $K3$ surface $S_b$, there is the distinguished $0$--cycle ${\mathfrak o}_{S_b}$ such that $c_2(S_b)=24 {\mathfrak o}_{S_b}$ \cite{BV}. Let $p_i\colon \Ss^{m/B}\to \Ss$, $i=1,\ldots,m$, denote the projections to the two factors. Let $T_{\Ss/B}$ denote the relative tangent bundle.
The assignment
  \[ \begin{split} \Pi_0^\Ss &:= (p_1)^\ast \bigl({1\over 24} c_2(T_{\Ss/B})\bigr) \ \ \ A^2(\Ss\times_B \Ss)\ ,\\
                         \Pi_4^\Ss &:= (p_2)^\ast \bigl({1\over 24} c_2(T_{\Ss/B})\bigr) \ \ \ A^2(\Ss\times_B \Ss)\ ,\\
                         \Pi_2^\Ss &:= \Delta_\Ss - \Pi_0^\Ss - \Pi_4^\Ss\\
                    \end{split}\]
          defines (by restriction) an MCK decomposition for each fibre, i.e.
          \[  \Pi_j^{S_b}:= \Pi_j^\Ss\vert_{S_b\times S_b}\ \ \ \in A^2(S_b\times S_b)\ \ \ (j=0,2,4) \]
          is an MCK decomposition for any $b\in B$ \cite[Example 8.17]{SV}.
          
  Next, we consider the $m$--fold relative fibre product $\Ss^{m/B}$. Let
    \[ p_{i,j}\colon \Ss^{2m/B}\ \to\ \Ss^{2/B} \ \ \ (1\le i<j\le 2m)\]
    denote projection to the $i$-th and $j$-th factor. We define
    \[  \begin{split}  \Pi_j^{\Ss^{m/B}} := {\displaystyle \sum_{k_1+k_2+\cdots+k_m=j}}  (p_{1,m+1})^\ast ( \Pi_{k_1}^{\Ss})\cdot (p_{2,m+2})^\ast (\Pi_{k_2}^\Ss)\cdot\ldots\cdot   
                                                                                       (p_{m,2m})^\ast ( \Pi_{k_m}^{\Ss})&\\           \ \ \ \in A^{2m}(\Ss^{4m/B})\ ,\ \ \ 
    (j=0,2,4,\ldots,4m)\ &.\\
    \end{split}\]
    By construction, the restriction to each fibre induces an MCK decomposition (the ``product MCK decomposition'')
    \[ \begin{split} \Pi_j^{(S_b)^m} :=  \Pi_j^{\Ss^{m/B}}\vert_{(S_b)^{2m}} = {\displaystyle \sum_{k_1+k_2+\cdots+k_m=j}}  \Pi_{k_1}^{S_b}\times \Pi_{k_2}^{S_b}\times\cdots
        \times \Pi_{k_m}^{S_b}\ \ \ \in A^{2m}((S_b)^{4m})\ ,&\\
        \ \ \ (j=0,2,4,\ldots,4m)\ .&\\
        \end{split}\]
    \end{proof}
    
    \begin{proposition}\label{prod2} Let $\Ss\to B$ be a family of $K3$ surfaces. There exist relative correspondences
    \[  \Theta_1\ ,\ldots,\ \Theta_m\in A^{2m}(\Ss^{m/B}\times_B \Ss)\ ,\ \ \ \Xi_1\ ,\ldots, \ \Xi_m\in  A^{2}(\Ss\times_B \Ss^{m/B})  \]
    such that for each $b\in B$, the composition
    \[  \begin{split}   A^{2m}_{(2)}\bigl((S_b)^m\bigr)\ \xrightarrow{((\Theta_1\vert_{(S_b)^{m+1}})_\ast,\ldots, (\Theta_m\vert_{(S_b)^{m+1}})_\ast)}\
             A^2(S_b)\oplus \cdots \oplus A^2(S_b)&\\
             \ \ \ \ \ \ \xrightarrow{((\Xi_1+\ldots+\Xi_m)\vert_{(S_b)^{m+1}})_\ast}\ A^{2m}\bigl((S_b)^m\bigr)&\\
             \end{split} \]
      is the identity.
    \end{proposition}
    
    \begin{proof} 
    
     As before, let 
    \[ p_{i,j}\colon\ \ \  \Ss^{2m/B}\ \to\ \Ss^{2/B} \ \ \ (1\le i<j\le 2m)\]
    denote projection to the $i$-th and $j$-th factor, and let 
    \[ p_i\colon \ \ \  \Ss^{m/B}\ \to\ \Ss \ \ \  (1\le i\le m) \]
    denote projection to the $i$--th factor.
            
        We now claim that for each $b\in B$, there is equality
       \begin{equation}\label{both} \begin{split} ( \Pi_{4m-2}^{\Ss^{m/B}})\vert_{(S_b)^{2m}} =  {1\over 24^{m-1}}\Bigl(  {}^t \Gamma_{p_{1}}\circ \Pi_2^\Ss\circ \Gamma_{p_{1}}\circ 
               \bigl(   (p_{1,m+1})^\ast (\Delta_\Ss )\cdot \prod_{{2\le j\le 2m}} (p_{j})^\ast c_2(T_{\Ss/B}) &\bigr)\\  
               + \ldots +
                           {}^t \Gamma_{p_{m}}\circ \Pi_2^\Ss\circ \Gamma_{p_{m}}\circ 
               \bigl(   (p_{m,2m})^\ast (\Delta_\Ss )\cdot \prod_{\stackrel{1\le j\le 2m-1}{j\not=m}}(p_{j})^\ast c_2(T_{\Ss/B})    \bigr)   \Bigr)&\vert_{(S_b)^{2m}}\\
               \ \ \ \hbox{in}\ A^{2m}((&S_b)^{m}\times (S_b)^{m})\ .\\
               \end{split}\end{equation}
         Indeed, using Lieberman's lemma \cite[16.1.1]{F}, we find that
         \[ \begin{split} ( {}^t \Gamma_{p_{1}}\circ &\Pi_2^\Ss\circ \Gamma_{p_{1}})\vert_{(S_b)^{2m}} = \bigl(({}^t \Gamma_{p_{1,m+1}})_\ast 
         (\Pi_2^{\Ss})\bigr)\vert_{(S_b)^{2m}} =
           \bigl((p_{1,m+1})^\ast (\Pi_2^{\Ss})\bigr)\vert_{(S_b)^{2m}}\ ,\\
                           &\vdots\\
           ( {}^t \Gamma_{p_{m}}\circ &\Pi_2^\Ss\circ \Gamma_{p_{m}})\vert_{(S_b)^{2m}} = \bigl(({}^t \Gamma_{p_{m,2m}})_\ast 
         (\Pi_2^{\Ss})\bigr)\vert_{(S_b)^{2m}} =
           \bigl((p_{m,2m})^\ast (\Pi_2^{\Ss})\bigr)\vert_{(S_b)^{2m}}\ .\\
           \end{split}           \]
           
       Let us now (by way of example) consider the first summand of the right--hand--side of (\ref{both}). For brevity, let
        \[ P\colon\ \ \  (S_b)^{3m}\ \to\ (S_b)^{2m} \]
        denote the projection on the first $m$ and last $m$ factors. Writing out the definition of composition of correspondences,
        we find that
       \[ \begin{split}     &{1\over 24^{m-1}}\Bigl(  {}^t \Gamma_{p_{1}}\circ \Pi_2^\Ss\circ \Gamma_{p_{1}}\circ 
               \bigl(   (p_{1,m+1})^\ast (\Delta_\Ss )\cdot \prod_{\stackrel{2\le j\le 2m}{j\not=m+1}} (p_{j})^\ast c_2(T_{\Ss/B}) \bigr)\Bigr)\vert_{(S_b)^{2m}} =\\
                & {1\over 24^{2m-2}}\Bigl(   \bigl((p_{1,m+1})^\ast (\Pi_2^{S_b})\bigr)   \circ 
               \bigl(   (p_{1,m+1})^\ast (\Delta_{S_b} )\cdot \prod_{{m+2\le j\le 2m}} (p_{j})^\ast c_2(T_{S_b}) \bigr)\Bigr) =\\ 
               & P_\ast    \Bigl( \bigl( (\Delta_{S_b})_{(1,m+1)} \times {\mathfrak o}_{S_b} \times\cdots\times{\mathfrak o}_{S_b} \times S_b\times\cdots\times S_b\bigr)\cdot \\
               &\ \ \ \ \ \ \ \ \bigl( S_b\times\cdots\times S_b\times (\Pi_2^{S_b})_{(m+1,2m+1)}\times S_b\times\cdots\times S_b     \bigr)  \Bigr)= \\
               & P_\ast \Bigl(  \bigl((\Delta_{S_b}\times S_b)\cdot (S_b\times\Pi_2^{S_b})\bigr)_{(1,m+1,2m+1)}\times  {\mathfrak o}_{S_b}\times\cdots\times {\mathfrak o}_{S_b}\times S_b\times\cdots\times S_b\Bigr)=\\
               & \Pi_2^{S_b}\times \Pi_4^{S_b}\times\cdots\times \Pi_4^{S_b}\ \ \ \ \ \ \hbox{in}\ A^{2m}\bigl( (S_b)^m\times (S_b)^m\bigr)\ .\\
               \end{split}\]  
               (Here, we use the notation $(C)_{(i, j)}$ to indicate that the cycle $C$ lies in the $i$th and $j$th factor, and likewise for $(D)_{(i,j,k)}$.)           
                          
      Doing the same for the other summands in (\ref{both}), one convinces oneself that both sides of (\ref{both}) are equal to the fibrewise product Chow--K\"unneth component
         \[  \Pi_{4m-2}^{(S_b)^m}=\Pi_2^{S_b}\times \Pi_4^{S_b}\times  \cdots\times\Pi_4^{S_b}  +\cdots  +      \Pi_4^{S_b}\times\cdots\times\Pi_4^{S_b}\times \Pi_2^{S_b}    \ \ \ \in A^{2m}((S_b)^m\times (S_b)^m)\ ,\]
         thus proving the claim.

  Let us now define
      \[ \begin{split}
           \Theta_i&:={1\over 24^{m-1}} \, \Gamma_{p_{i}}\circ 
               \bigl(   (p_{i,m+i})^\ast (\Delta_\Ss )\cdot \prod_{\stackrel{j\in [m+2,2m]}{ j\not\in\{i,m+i\}}} (p_{j})^\ast c_2(T_{\Ss/B})    \bigr)\ \ \ \in A^{2m}((\Ss^{m/B})\times_B \Ss)\ ,\\  
                 \Xi_i&:= {}^t \Gamma_{p_{i}}\circ \Pi_2^\Ss\ \ \ \ \ \ \in A^2(\Ss\times_B (\Ss^{m/B})) \ ,\\
                       \end{split}\]
                       where $1\le i\le m$.
    It follows from equation (\ref{both}) that there is equality 
      \begin{equation}\label{transp} \begin{split} \Bigl( (\Xi_1\circ \Theta_1 + \cdots +\Xi_m\circ \Theta_m)\vert_{(S_b)^{2m}}   \Bigr){}_\ast =
      \bigl(\Pi_{4m-2}^{(S_b)^m}\bigr){}_\ast\colon &\\
           \ \ A^{i}_{(j)}\bigl((S_b)^m\bigr)\ \to\ A^{i}_{(j)}\bigl((S_b)^m\bigr)&\ \ \ \forall b\in B\ \ \ \forall (i,j)\ .\\
      \end{split}\end{equation}      
      Taking $(i,j)=(2m,2)$, this proves the proposition.

         \end{proof}

 The following is a version of proposition \ref{prod2} for the group $A^2_{(2)}((S_b)^m)$:     
  
 \begin{proposition}\label{prod3} Let $\Ss\to B$ be a family of $K3$ surfaces. There exist relative correspondences
    \[  \Theta^\prime_1\ ,\ldots,\ \Theta^\prime_m\in A^{2m}(\Ss\times_B (\Ss^{m/B}))\ ,\ \ \ \Xi^\prime_1\ ,\ldots, \ \Xi^\prime_m\in  A^{2}( (\Ss^{m/B})\times_B \Ss)  \]
    such that for each $b\in B$, the composition
    \[  \begin{split}   A^{2}_{(2)}\bigl((S_b)^m\bigr)\ \xrightarrow{((\Xi^\prime_1\vert_{(S_b)^{m+1}})_\ast,\ldots, (\Xi^\prime_m\vert_{(S_b)^{m+1}})_\ast)}\
             A^2(S_b)\oplus \cdots \oplus A^2(S_b)&\\
             \ \ \ \ \ \ \xrightarrow{((\Theta^\prime_1+\ldots+\Theta^\prime_m)\vert_{(S_b)^{m+1}})_\ast}\ A^{2}\bigl((S_b)^m\bigr)&\\
             \end{split} \]
      is the identity.
  \end{proposition} 
  
  \begin{proof} One may take
    \[ \begin{split}   \Theta^\prime_i&:= {}^t \Theta_i\ \ \ \in\ A^{2m}(\Ss\times_B (\Ss^{m/B}))\ ,\\
                             \Xi^\prime_i&:= {}^t \Xi_i   \ \ \ A^2((\Ss^{m/B})\times_B \Ss)\ \ \ \ (i=1,\ldots,m)\ .\\
                    \end{split}\]
       By construction, the product MCK decomposition $\{ \Pi_i^{(S_b)^m}\}$ satisfies
          \[ \Pi_2^{(S_b)^m} = {}^t \bigl(\Pi_{4m-2}^{(S_b)^m}\bigr)\ \ \ \hbox{in}\ A^{2m}\bigl( (S_b)^m\times (S_b)^m\bigr)\ .\]   
     Hence, the transpose of equation (\ref{transp}) gives the equality
        \[    \begin{split}    \bigl( \Pi_2^{(S_b)^m} \bigr){}_\ast = \bigl( {}^t (\Pi_{4m-2}^{(S_b)^m})\bigr){}_\ast = \bigl( {}^t \Theta_1\circ {}^t \Xi_1+\ldots+{}^t \Theta_m\circ {}^t \Xi_m\bigr){}_\ast\colon&\\\ \  \ A^{i}_{(j)}\bigl((S_b)^m\bigr)\ \to\ A^{i}_{(j)}\bigl((S_b)^m\bigr)\ \ &\ \forall b\in B\ \ \ \forall (i,j)\ .\\
        \end{split} \]
        Taking $(i,j)=(2,2)$, this proves the proposition.
         \end{proof}

 \subsection{Spread}
 \label{ssvois}
 
    \begin{lemma}[Voisin \cite{V0}, \cite{V1}]\label{projbundle} Let $M$ be a smooth projective variety of dimension $n+r$, and let $L_1,\ldots,L_r$ be very ample line bundles on $M$. Let 
    \[  \XX\to B\]
   denote the universal family of codimension $r$ smooth complete intersections $X_b\subset M$ of type
   \[ X_b=M\cap D_1\cap\cdots \cap D_r\ ,\ \ \ D_i\in \vert L_i\vert\ ,\ \ i=1,\ldots r\ .\]
   (That is,
      \[  B\subset\vert L_1\vert\ \times \cdots \times \vert L_r\vert  \] 
      is a Zariski open.)
      Let
   \[   p\colon \wt{\XX\times_B \XX}\ \to\ \XX\times_B \XX\]
   denote the blow--up of the relative diagonal. 
 Then $\wt{\XX\times_B \XX}$ is Zariski open in $V$, where $V$ is a fibre bundle over $\wt{M\times M}$, the blow--up of $M\times M$ along the diagonal, and the fibres of $V\to\wt{M\times M}$ are products of projective spaces.
   \end{lemma} 
  
  \begin{proof} This is \cite[Proof of Proposition 3.13]{V0} or \cite[Lemma 1.3]{V1}. The idea is to define $V$ as
   \[  V:=\Bigl\{ \bigl((x,y,z),\sigma\bigr) \ \vert\ \sigma\vert_z=0\Bigr\}\ \ \subset\ \wt{M\times M}\times \vert L\vert\ .\]
   The very ampleness assumption ensures that $V\to\wt{M\times M}$ is a projective bundle.
    \end{proof}

  This is used in the following key proposition: 
   
 \begin{proposition}[Voisin \cite{V1}]\label{propvois} Let $M$ be a smooth projective variety of dimension $n+r$, 
 and suppose that
   \[ A^\ast_{hom}(M)=0\ .\]
Let $L_1,\ldots,L_r$ be very ample line bundles on $M$, and let 
  $ \XX\ \to\ B$ be as in lemma \ref{projbundle}.

    Assume $\Gamma\in A^n(\XX\times_B \XX)$ is such that the restriction
   \[ \Gamma_b:= \Gamma\vert_{X_b\times X_b}\ \ \ \in\ A^n(X_b\times X_b) \]
   is homologically trivial, for very general $b\in B$. Then there exists $\delta\in A^n(M\times M)$ such that
   \[ \Gamma_b + \delta_b=0\ \ \ \hbox{in}\ A^n(X_b\times X_b)\ \ \ \forall b\in B\ .\]
  \end{proposition}
 
 \begin{proof} This follows from \cite[Proposition 1.6]{V1}. (NB: The result \cite[Proposition 1.6]{V1} is stated only for hypersurfaces, i.e. $r=1$. However, as noted in \cite[Remark 0.7]{V1}, the complete intersection case follows from this.)
 
 In the special case $n=2$ (which is the only case we will need in this note), proposition \ref{propvois} is already contained in \cite{V0}. Indeed, the Leray spectral sequence argument \cite[Lemmas 3.11 and 3.12]{V0} gives the existence of $\delta\in A^2(M\times M)$ such that (after shrinking the base $B$)
   \[  \Gamma + \delta\vert_{\XX\times_B \XX}=0\ \ \ \hbox{in}\ H^4(\XX\times_B \XX)\ .\]
   But using lemma \ref{projbundle} (plus some basic properties of varieties with trivial Chow groups, cf. \cite[Section 3.1]{V0}), one finds that
    \[ A^2_{hom}(\XX\times_B \XX)=0\ . \]
   Therefore, we must have
   \[  \Gamma + \delta\vert_{\XX\times_B \XX}=0\ \ \ \hbox{in}\ A^2(\XX\times_B \XX)\ .\]
   In particular, this implies that
    \[ \Gamma_b + \delta_b=0\ \ \ \hbox{in}\ A^n(X_b\times X_b)\ \ \ \hbox{for\ general\  $b\in B$}\ .\]
    To obtain the result for {\em all\/} $b\in B$, one can invoke \cite[Lemma 3.2]{Vo}.
    \end{proof}

\subsection{Mukai models}

\begin{theorem}[Mukai \cite{Muk}]\label{muk} Let $S$ be a general $K3$ surface of degree $10$ (i.e. genus $g(S)=6$). Let $G=G(2,5)$ denote the Grassmannian of lines in $\PP^4$. Then $S$ is isomorphic to the zero locus of a section of $\OO_G(1)^{\oplus 3}\oplus \OO_G(2)$.
 \end{theorem}

\begin{remark}\label{family} Let
  \[ \Ss\ \subset\ G\times B \]
   denote the universal family of smooth codimension $3$ complete intersections defined by $\OO_G(1)^{\oplus 3}\oplus \OO_G(2)$, where
  \[ B\ \subset\ \PP H^0\bigl(G,\OO_G(1)\bigr)^{\times 3}\times \PP H^0\bigl(G,\OO_G(2)\bigr) \]  
  is the Zariski open parametrizing smooth surfaces $S_b\subset G$.
  We will refer to the family
  \[ \Ss\ \to\ B \]
  as the {\em universal family of degree $10$ $K3$ surfaces\/}.
\end{remark}

 \subsection{EPW cubes}
 \label{sscube}
 
 \begin{definition}[Iliev--Kapustka--Kapustka--Ranestad \cite{IKKR}]\label{def} 
 Let $W$ be a complex vector space of dimension $6$ equipped with a skew--symmetric form 
   \[ \nu\colon\ \ \ \wedge^3 W \times \wedge^3 W\ \to\ \C\ .\]
   Let $LG_\nu$ denote the variety of $10$--dimensional subspaces in $\wedge^3 W$ that are Lagrangian with respect to $\nu$. For any $3$--dimensional subspace
   $U\in G(3,W)$, the $10$--dimensional subspace
     \[ T_U := \wedge^2 U \wedge W\ \subset\ \wedge^3 W \]
     is in $LG_\nu$.
     
     Given $A\in LG_\nu$ and $k\in\NN$, define the degenerary locus
     \[ D_k^A := \bigl\{  U \in G(3,W)\ \vert\ \dim (A\cap T_U)\ge k\bigr\}\ \ \subset\ G(3,W)\ .\]
     The scheme $D_2^A$ is called an {\em EPW cube\/}. For $A$ generic, the EPW cube $D_2^A$ is of dimension $6$, and $\hbox{Sing}(D_2^A)=D_3^A$ is a smooth threefold.
     \end{definition}

 \begin{theorem}[Iliev--Kapustka--Kapustka--Ranestad \cite{IKKR}]\label{ikkr} Notation as in definition \ref{def}.
 
 \noindent
 (\rom1)
 There is a Zariski open $LG_\nu^1\subset LG_\nu$ with the following property: for any $A\in LG_\nu^1$, there exists a double cover
   \[ Y_A\ \to\ D_2^A \]
   branched along $D_3^A$, and $Y_A$ is a hyperk\"ahler variety.
   
   \noindent
   (\rom2)
   There is a divisor $\Delta^1\subset LG_\nu^1$ such that for general $A\in \Delta^1$, the variety $Y_A$ is birational to the Hilbert scheme $(S_A)^{[3]}$ for some degree $10$ $K3$ surface $S_A$.
   
  \noindent
  (\rom3) Given a generic degree $10$ $K3$ surface $S$, there exists $A\in\Delta^1$ such that $S=S_A$. 
    \end{theorem}
    
  \begin{proof} Point (\rom1) is contained in \cite[Theorem 1.1]{IKKR}. 
  
  Point (\rom2) is \cite[Section 5]{IKKR}. (NB: the divisor that we denote $\Delta^1$ is written as
  $\Delta\setminus (\Gamma\cup \Sigma)$ in \cite{IKKR}.)

  For point (\rom3), we note that the construction of $S_A$ for general $A\in\Delta$ in \cite[Section 4]{IKKR} is modelled on O'Grady's construction in \cite[Section 4.1]{OG}; point (\rom3) thus follows from O'Grady's result \cite[Proposition 4.14]{OG}.
  \end{proof}  
    
  \begin{remark}\label{anti} As noted in \cite{IKKR}, a noteworthy consequence of theorem \ref{ikkr}(\rom2) is that double EPW cubes $Y_A$ are of $K_3^{[3]}$ type. 
  
  Theorem \ref{ikkr}(\rom3) implies that the covering involution
    \[ \iota_A\colon\ \ \ Y_A\ \to\ Y_A \]
    is anti--symplectic: indeed (as noted in \cite{IKKR}), if it were symplectic the fixed--locus would be a symplectic subvariety, whereas the fixed--locus of $\iota_A$ is the inverse image of $D_3^A$ which is of dimension $3$.
  
  Theorem \ref{ikkr}(\rom3) implies that if $S$ is a generic degree $10$ $K3$ surface, there exists an anti--symplectic birational involution
    \[ \iota\colon\ \ \ S^{[3]}\ \dashrightarrow\ S^{[3]}\ .\]
    I do not know whether there is a geometric interpretation of the involution $\iota$, similar to the geometric interpretation of the birational involution
    \[   \iota\colon\ \ \ S^{[2]}\ \dashrightarrow\ S^{[2]}\ \]
    related to double EPW sextics given in \cite[Section 4.3]{OG0}.
    \end{remark}
    
  

We now translate some of the results of \cite{IKKR} into statements that will be convenient for the purposes of this note:

\begin{proposition}\label{families} Let $\Delta^1\subset LG^1_\nu$ be the divisor of theorem \ref{ikkr}.  Let $\TT\to \MM_6$ be the universal genus $6$ $K3$ surface over the moduli space $\MM_6$. 

\noindent
(\rom1) There exist projective morphisms
  \[ \XX_{\Delta^1}\ \xrightarrow{\rho}\  \DD_{\Delta^1}  \ \xrightarrow{\pi}\ \Delta^1\ ,\]
  such that for each $A\in\Delta^1$, the fibre $X_A:= (\pi\circ\rho)^{-1}(A)$ is a double EPW cube, and the fibre $D_A:=\pi^{-1}(A) $ is an EPW cube.
  
 \noindent
 (\rom2) Let $\TT\to \MM_6$ be the universal genus $6$ $K3$ surface over the moduli space $\MM_6$, and let $\TT^{[3]}\to\MM_6$ denote the universal Hilbert cube. 
There exist Zariski opens $\Delta^{1,0}\subset \Delta^1$, and $\MM^0_6\subset \MM_6$, and a generically $2:1$ rational map
  \[ \Psi\colon\ \ \  \TT^{[3]/\MM^0_6}\ \dashrightarrow\  \EE  \ .\]
  Here, $\TT^{[3]/\MM_6^0}:= (\TT^{[3]})\times_{\MM_6} \MM^0_6$, and $\EE$ is the quotient stack
   \[ \EE:= \DD_{\Delta^{1,0}}/P\ ,\]
   where $\DD_{\Delta^{1,0}}:=\DD_{\Delta^1}\times_{\Delta^1} \Delta^{1,0}$, and 
   $P:=PGL(W)$ acts on $\Delta^{1,0}$ and on $G=G(3,W)$. The map $\Psi$ fits into a diagram
    \[ \begin{array}[c]{ccccccccc}
      \Ss^{3/B^0} &\stackrel{}{\dashrightarrow}&\ \ \ \ \ \  \TT^{[3]/\MM_6^0}& &&& \XX_{\Delta^{1,0}}&\hookrightarrow& \XX\\
      &&& \ \ \ \ \ \ \searrow{\scriptstyle \Psi}& &&\downarrow&&\downarrow\\    
      \downarrow &  & \downarrow &  &            \EE := \DD_{\Delta^{1,0}}/P \ \  & \leftarrow     & \DD_{\Delta^{1,0}} & \hookrightarrow  &\DD \\
       &&  && \downarrow && \downarrow&&\downarrow\\
      B^0 & \to& \MM_6^0 & \stackrel{f}{\to} & \MM_{\Delta^{1,0}}=\Delta^{1,0}/P & \leftarrow & \Delta^{1,0}  & \hookrightarrow &LG^1_\nu     \\
      \end{array} \]
      
      Here, $\MM_{\Delta^{1,0}}$ is the image of $\Delta^{1,0}$ under the period map to the moduli space, and $\MM_{\Delta^{1,0}}$ is a geometric quotient
        \[ \MM_{\Delta^{1,0}}= \Delta^{1,0}/P\ .\]
        The morphism $f$ is an isomorphism.
      (And $\Ss\to B$ is the universal family of remark \ref{family}, and $B^0\subset B$ is a Zariski open).
      
    \noindent
    (\rom3) The quotient stack $\EE$ is a  Deligne--Mumford stack, and so
      \[ A^i(\EE)\cong A^i_P(\DD_{\Delta^{1,0}})\ ,\]
      where the right--hand side denotes equivariant Chow groups, in the sense of Edidin--Graham \cite{EG}.
    \end{proposition}
  
  \begin{proof}    
 {}
 
 \noindent
 (\rom1) There exists a tower of projective morphisms
   \[ \XX\ \to\ \DD\ \to\ LG^1_\nu\ ,\]
   where a fibre $D_A$ is an EPW cube, and a fibre $X_A$ is a double EPW cube, and $\XX\to LG^1_\nu$ is smooth \cite[Section 5]{IKKR}. By base change, one obtains
   \[ \XX_{\Delta^1}\ \to\ \DD_{\Delta^1} \ \to\ \Delta^1\ .\]
   
 \noindent
 (\rom2) First, we note that (as proven in \cite{OG}) for a given $A\in\Delta^{1,0}$, the associated $K3$ surface $S_A$ is well--defined up to projectivities, and so there is a map $\Delta^{1,0}\to \MM_6$. Conversely, given a general genus $6$ $K3$ surface $S$, the element $A\in\Delta^1$ such that $S=S_A$ is well--defined up to the action of $P=PGL(W)$. This proves that $f$ is an isomorphism on appropriate opens.
 
 To construct $\EE$, we note that $\DD_{\Delta^{1,0}}$ is defined as
   \[ \DD_{\Delta^{1,0}}:= \bigl\{  (U,A)\ \vert\ U\in D_2^A        \bigr\}\ \ \ \subset\  G\times\Delta^{1,0}  \ ,\]
   and so $P$ acts naturally on $\DD_{\Delta^{1,0}}$. 
   
   The map $\Psi$ is defined by sending a generic point $x\in (S_b)^{[3]}$ to 
   \[     \rho\bigl(   (\phi_b)(x)\bigr)  \ \ \ \in\ D_2^{f(b)}    \ ,\]
   where $\phi_b\colon (S_b)^{[3]}\dashrightarrow X_{f(b)}$ is the birational map of theorem \ref{ikkr}, and $\rho\colon X_{f(b)}\to D_2^{f(b)}$ is the double cover.   

 \noindent
 (\rom3) Let $s\colon \DD_{\Delta^{1,0}}\to \Delta^{1,0}$ denote the projection. The stabilizer of a point $e\in\DD_{\Delta^{1,0}}$ for the action of $P$ is contained in the stabilizer of $s(e)$ for the $P$--action on $\Delta^{1,0}$. This stabilizer is finite, since $\Delta^{1,0}$ is contained in $LG^1_\nu$, which is contained in the stable locus \cite{OG06}. 
 
 The statement about the Chow group of $\EE$ follows from this. (For any Deligne--Mumford stack, Chow groups with rational coefficients have been defined 
 \cite{Gil}, \cite{Vis}. These Chow groups agree with the equivariant Chow groups \cite{EG}.) 
 
 \end{proof}

 \begin{corollary}\label{mck} Let $A\in\Delta^1$ be general, and let $X=X_A$ be the associated double EPW cube. Then $X$ has an MCK decomposition, and the Chow ring of $X$ has a bigrading $A^\ast_{(\ast)}(X)$ with $A^i_{(j)}(X)=0$ if $j>i$ and  $A^i_{(j)}(X)=0$ if $j$ is odd. 
 \end{corollary}
 
 \begin{proof} The variety $X$ is birational to a Hilbert cube $(S_A)^{[3]}$ (theorem \ref{ikkr}(\rom2)). Hilbert cubes of $K3$ surfaces have an MCK decomposition (theorem \ref{charles}). It follows from lemma \ref{hk} that $X$ has an MCK decomposition, and that there is an isomorphism of bigraded rings
   \[  A^\ast_{(\ast)}(X)\cong A^\ast_{(\ast)}((S_A)^{[3]})\ .\] 
   The vanishing $A^i_{(j)}(X)=0$ for $j>i$ and for $j$ odd follows from the corresponding property for $(S_A)^{[3]}$. 
 \end{proof}

 \section{Hard Lefschetz}
 
 In this section, we prove a ``hard Lefschetz type'' isomorphism for Chow groups of certain varieties. This hard Lefschetz result (and in particular, the version for double EPW cubes, corollary \ref{hardepw}) will be a crucial ingredient in the proof of the main result of this note (theorem \ref{main}).
 
 
 \begin{theorem}\label{hard} Let $\Ss\to B$ be the universal family of $K3$ surfaces of degree $10$ (cf. remark \ref{family}). Let $L\in A^1(\Ss^{m/B})$ be a line bundle such that the restriction $L_b$ (to the fibre over $b\in B$) is big for very general $b\in B$.
   Then
     \[  \cdot (L_b)^{2m-2}\colon\ \ \ A^2_{(2)}((S_b)^m)\ \to\ A^{2m}_{(2)}((S_b)^m) \]
     is an isomorphism for all $b\in B$.
       \end{theorem}
      
   \begin{proof} This is proven using the technique of spread as developed by Voisin \cite{V0}, \cite{V1}.
   Let us write
    \[ \Gamma_{L^{2m-2}}:=  (p_1)^\ast(L^{2m-2})\cdot \Delta_{\Ss^{m/B}}\ \ \ \in A^{4m-2}\bigl( (\Ss^{m/B})\times_B (\Ss^{m/B})\bigr)\ ,\]
    where 
     \[ \Delta_{\Ss^{m/B}}\ \subset\    (\Ss^{m/B})\times_B (\Ss^{m/B})\] 
     is the relative diagonal, and
     \[ p_1\colon\ \ \  (\Ss^{m/B})\times_B (\Ss^{m/B})\ \to\ \Ss^{m/B} \]
     is projection on the first factor. The relative correspondence $ \Gamma_{L^{2m-2}}$ acts on Chow groups as multiplication by $L^{2m-2}$.
     
     As ``input'', we will make use of the following result:
     
     \begin{proposition}[L. Fu \cite{LFu0}]\label{input} Let $X$ be a smooth projective variety of dimension $n$ verifying the Lefschetz standard conjecture $B(X)$. Let $L\in A^1(X)$
      be a big line bundle. Then
        \[ \cup L^{n-2}\colon\ \ \ H^2(X)/N^1 H^2(X)\ \to\ H^{2n-2}(X)/N^{n-1}H^{2n-2}(X) \]
        is an isomorphism. (Here $N^\ast$ denotes the coniveau filtration \cite{BO}, so $N^i H^{2i}(X)$ is the image of the cycle class map.)
        Moreover, there is a correspondence $C\in A^2(X\times X)$ inducing the inverse isomorphism.
        \end{proposition}
        
        \begin{proof} This follows from the proof of \cite[Theorem 4.11]{LFu0}. Alternatively, here is an explicit argument: it follows from \cite[Lemma 3.3]{LFu0} that
          \[  \cup L^{n-2}\colon\ \ \ H^2(X)/N^1 H^2(X)\ \to\ H^{2n-2}(X)/N^{n-1}H^{2n-2}(X) \]
        is an isomorphism. Since the category of motives for numerical equivalence $\MM_{\rm num}$ is semisimple \cite{J0}, it follows that there is an isomorphism of motives
        \[   h^2(X)\oplus \bigoplus_i \LLL(m_i)\ \cong\  h^{2n-2}(X)(n-2)\oplus \bigoplus_j \LLL(m_j)\ \ \ \hbox{in}\ \MM_{\rm num}\ ,\]
        where the arrow from $h^2(X)$ to $h^{2n-2}(X)(n-2)$ is given by $\Gamma_{L^{n-2}}\in A^{2n-2}(X\times X)$, and $\LLL$ denotes the Lefschetz motive.
        Since homological and numerical equivalence coincide for $X$ and for $\LLL$, this implies there is also an isomorphism
        \[   h^2(X)\oplus \bigoplus_i \LLL(m_i)\ \cong\  h^{2n-2}(X)(n-2)\oplus \bigoplus_j \LLL(m_j)\ \ \ \hbox{in}\ \MM_{\rm hom}\ ,\]
        with the arrow from $h^2(X)$ to $h^{2n-2}(X)(n-2)$ being given by $\Gamma_{L^{n-2}}$. It follows that there exists a correspondence $C$ as required.
                     \end{proof}

     Any fibre $(S_b)^m$ of the family $\Ss^{m/B}\to B$ verifies the Lefschetz standard conjecture (the Lefschetz standard conjecture is known for products of surfaces). Applying proposition \ref{input}, this means that for all $b\in B$ there exists a correspondence
      \[ C_b\ \ \in A^{2}\bigl(   (S_b)^m\times     (S_b)^m\bigr) \]
     with the property that the compositions
      \[  
      H^2\bigl(  (S_b)^m\bigr)/N^1 \ \xrightarrow{\cdot    (L_b)^{2m-2}}\ H^{4m-2}\bigl( (S_b)^{m}\bigr)/N^{2m-1}\ \xrightarrow{ (C_b)_\ast} H^2\bigl( (S_b)^m\bigr)/N^1 \]
      and
       \[  
      H^{4m-2}\bigl(  (S_b)^m\bigr)/N^{2m-1} \ \xrightarrow{(C_b)_\ast}\ H^{2}\bigl( (S_b)^{m}\bigr)/N^1\ \xrightarrow{\cdot    (L_b)^{2m-2} } H^{4m-2}\bigl( (S_b)^m
      \bigr)/N^{2m-1} \]
             are the identity. In other words, for all $b\in B$ there exist 
             \[\gamma_b\ ,\ \ \  \gamma_b^\prime\in 
             A^{2m}\bigl(   (S_b)^m\times     (S_b)^m\bigr)\] 
             supported on $D_b\times D_b\subset    (S_b)^m\times     (S_b)^m$ for some divisor $D_b\subset (S_b)^m$ and such that
      \[   \begin{split}   \Pi_2^{\Ss^{m/B}}\vert_{(S_b)^m}    \circ C_b\circ \bigl( (\Pi_{4m-2}^{\Ss^{m/B}}\circ \Gamma_{L^{2m-2}}\circ \Pi_2^{\Ss^{m/B}})\vert_{(S_b)^m}\bigr)  &= 
          \Pi_2^{\Ss^{m/B}}\vert_{(S_b)^m} +\gamma_b\  
      \ ,\\
              \Pi_{4m-2}^{\Ss^{m/B}}\vert_{(S_b)^m}    \circ     \bigl( \Gamma_{L^{2m-2}}       \circ  (\Pi_{2}^{\Ss^{m/B}}\bigr)\vert_{(S_b)^m}\circ  C_b \circ \bigl(\Pi_{4m-2}^{\Ss^{m/B}})\vert_{(S_b)^m}\bigr)  &= \Pi_{4m-2}^{\Ss^{m/B}}\vert_{(S_b)^m} +\gamma_b^\prime\\
                 &\ \ \ \ \ \ \hbox{in}\ H^{4m}\bigl( (S_b)^m\times (S_b)^m\bigr)\ .\\  
                       \end{split}    \]
    Applying a Hilbert schemes argument as in \cite[Proposition 3.7]{V0} (cf. also \cite[Proposition 2.10]{excubic4}), we can find a relative correspondence 
     \[ \Cc\in A^2\bigl(  (\Ss^{m/B})\times_B (\Ss^{m/B})\bigr) \]
     doing the same job as the various $C_b$, i.e. such that for all $b\in B$ one has
      \[ \begin{split}  ( \Pi_2^{\Ss^{m/B}}\circ \Cc\circ \Pi_{4m-2}^{\Ss^{m/B}}\circ \Gamma_{L^{2m-2}}\circ \Pi_2^{\Ss^{m/B}})\vert_{(S_b)^m}  &= \Pi_2^{\Ss^{m/B}}\vert_{(S_b)^m} +\gamma_b\ ,\\
            ( \Pi_{4m-2}^{\Ss^{m/B}}\circ\Gamma_{L^{2m-2}}\circ \Pi_{2}^{\Ss^{m/B}}\circ \Cc\circ \Pi_{4m-2}^{\Ss^{m/B}})\vert_{(S_b)^m}  &= \Pi_{4m-2}^{\Ss^{m/B}}\vert_{(S_b)^m}+\gamma_b^\prime\\
           &\ \ \ \ \ \  \ \hbox{in}\ H^{4m}\bigl( (S_b)^m\times (S_b)^m\bigr)\ .
            \\  \end{split}       \]   
     Applying once more the same Hilbert schemes argument \cite[Proposition 3.7]{V0}, we can also find a divisor $\DD\subset \Ss^{m/B}$ and relative correspondences
     \[ \gamma\ ,\ \ \ \gamma^\prime\ \ \ \in A^{2m}\bigl( \Ss^{m/B}\times_B \Ss^{m/B}\bigr) \]
     supported on $\DD\times_B \DD$ and doing the same job as the various $\gamma_b$, resp. $\gamma_b^\prime$. That is, $\gamma$ and $\gamma^\prime$ are such that for all $b\in B$ one has
      \[ \begin{split}  (\Pi_2^{\Ss^{m/B}}\circ \Cc\circ \Pi_{4m-2}^{\Ss^{m/B}}\circ \Gamma_{L^{2m-2}}\circ \Pi_2^{\Ss^{m/B}})\vert_{(S_b)^m}  &= (\Pi_2^{\Ss^{m/B}}+\gamma)\vert_{(S_b)^m}\ ,\\
            ( \Pi_{4m-2}^{\Ss^{m/B}}\circ \Gamma_{L^{2m-2}}\circ \Pi_{2}^{\Ss^{m/B}}\circ \Cc\circ \Pi_{4m-2}^{\Ss^{m/B}})\vert_{(S_b)^m}  &= (\Pi_{4m-2}^{\Ss^{m/B}}+\gamma^\prime)
            \vert_{(S_b)^m} \\ 
        &\ \ \ \ \ \ \     \ \hbox{in}\ H^{4m}\bigl( (S_b)^m\times (S_b)^m\bigr)\ .
            \\  \end{split}       \]                
           
      We now make an effort to rewrite this more compactly: the relative correspondences defined as
      \begin{equation}\label{defg}\begin{split} \Gamma&:=  \Pi_2^{\Ss^{m/B}}\circ \Cc\circ \Pi_{4m-2}^{\Ss^{m/B}}\circ \Gamma_{L^{2m-2}}\circ \Pi_2^{\Ss^{m/B}}  -     \Pi_2^{\Ss^{m/B}} -\gamma \ ,\\
              \Gamma^\prime&:= \Pi_{4m-2}^{\Ss^{m/B}}\circ \Gamma_{L^{2m-2}}\circ \Pi_{2}^{\Ss^{m/B}}\circ \Cc\circ \Pi_{4m-2}^{\Ss^{m/B}}  -     \Pi_{4m-2}^{\Ss^{m/B}} -\gamma^\prime\ \ \ \in A^{2m}  
              \bigl( (\Ss^{m/B})
      \times_B  (\Ss^{m/B})\bigr)  \\
      \end{split}        \end{equation}
      have the property that their restriction to any fibre is homologically trivial. That is, writing
      \[ \begin{split}  \Gamma_b &:= \Gamma\vert_{(S_b)^m\times (S_b)^m} \\
         \Gamma^\prime_b&:=  (\Gamma^\prime)\vert_{(S_b)^m\times (S_b)^m} \ \ \ \in A^{2m}_{}\bigl(  (S_b)^m\times (S_b)^m\bigr)\\
         \end{split}\]
         for the restriction to a fibre, we have that     
       \begin{equation}\label{hom0}  \Gamma_b\ ,\ \ \ \Gamma_b^\prime \ \ \ \in A^{2m}_{hom}\bigl(  (S_b)^m\times (S_b)^m\bigr)\ \ \ \forall b\in B\ ,\end{equation}
       
      Let us now define the modified relative correspondences
      \[ \begin{split}   \Gamma_1 &:=\Pi_2^{\Ss^{m/B}}\circ \Gamma\circ \Pi_2^{\Ss^{m/B}}\ ,\\ 
                               \Gamma_1^\prime&:= \Pi_{4m-2}^{\Ss^{m/B}}\circ \Gamma^\prime\circ  \Pi_{4m-2}^{\Ss^{m/B}}\ \ \ \in A^{2m} \bigl( \Ss^{m/B}\times_B \Ss^{m/B}\bigr)\ .\\
                              \end{split}\]
                              
        This modification does not essentially modify the fibrewise rational equivalence class: we have
        
        \begin{equation}\label{modif} \begin{split}   
                   (\Gamma_1)_b     &= \Gamma_b + (\gamma_1)_b\ ,\\
                 (\Gamma^\prime_1)_b     &= (\Gamma^\prime)_b + (\gamma^\prime_1)_b\ \ \ \hbox{in}\  A^{2m}\bigl( (S_b)^m\times (S_b)^m\bigr)\ ,\\
                 \end{split}\end{equation}
                 where $\gamma_1, \gamma_1^\prime\in A^{2m}\bigl(\Ss^{m/B}\times_B \Ss^{m/B}\bigr)$ are relative correspondences supported on $\DD\times_B \DD$.
                 (Indeed, this is true because $(\Pi_i^{(S_b)^m})^{\circ 2}=\Pi_i^{(S_b)^m}$ for all $i$, and the relative correspondences 
                 \[\Pi_2^{\Ss^{m/B}}\circ \gamma\circ \Pi_2^{\Ss^{m/B}}\ , \ \ \Pi_{4m-2}^{\Ss^{m/B}}\circ \gamma^\prime\circ \Pi_{4m-2}^{\Ss^{m/B}}\] 
                 are still supported on $\DD\times_B \DD$.)       
                 
     As $\Gamma$ and $\Gamma^\prime$ were fibrewise homologically trivial (equation (\ref{hom0})), the same is true for $\Gamma_1$ and $\Gamma_1^\prime$:
     
     \begin{equation}\label{hom1}  (\Gamma_1)_b\ ,\ \ \ (\Gamma_1^\prime)_b \ \ \ \in A^{2m}_{hom}\bigl(  (S_b)^m\times (S_b)^m\bigr)\ \ \ \forall b\in B\ ,\end{equation}
       
We now proceed to upgrade (\ref{hom1}) to a statement concerning the action on Chow groups:

\begin{claim}\label{rat1} We have
  \[ \begin{split}  \bigl((\Gamma_1)_b\bigr){}_\ast &=0 \colon\ \ \  A^i_{hom}\bigl((S_b)^m\bigr)  \ \to\  A^i_{hom}\bigl((S_b)^m\bigr)\ \ \ \forall b\in B\ ,\\
                            \bigl((\Gamma^\prime_1)_b\bigr){}_\ast  &=0 \colon \ \ \  A^i_{hom}\bigl((S_b)^m\bigr)  \ \to\  A^i_{hom}\bigl((S_b)^m\bigr)\ \ \ \forall b\in B\ .\\
                            \end{split}\]
                            \end{claim}
     
Let us prove claim \ref{rat1} for $\Gamma_1$ (the argument for $\Gamma_1^\prime$ is only notationally different). Using proposition \ref{prod3}, one finds there is a fibrewise equality modulo rational equivalence
  \begin{equation}\label{same} (\Gamma_1)_b = \Bigl( (\sum_{i=1}^m \Xi_i\circ \Theta_i)\circ \Gamma\circ (\sum_{i=1}^m  \Xi_i\circ \Theta_i ) \Bigr){}_b\ \ \ \hbox{in}\     A^{2m}\bigl( (S_b)^m\times (S_b)^m\bigr)\ \ \ \forall b\in B\ .\end{equation}
  To rewrite this, let us define relative correspondences
  \[ \Gamma_{k,\ell}:= \Theta_k\circ \Gamma\circ \Xi_\ell\ \ \ \in A^{2}\bigl( \Ss^{}\times_B \Ss^{}\bigr)\ \ \ ( 1\le k,\ell\le m)\ .\]
  With this notation, equality (\ref{same}) becomes the equality
  \begin{equation}\label{same2} (\Gamma_1)_b = \Bigl(  \sum_{k=1}^m \sum_{\ell=1}^m \Xi_k\circ \Gamma_{k,\ell}\circ \Theta_\ell\Bigr){}_b    \ \ \ \hbox{in}\     A^{2m}\bigl( (S_b)^m\times (S_b)^m\bigr)\ \ \ \forall b\in B\ .\end{equation}
  
  As $\Gamma$ is fibrewise homologically trivial (equation (\ref{hom0})), the same is true for the various $\Gamma_{k,\ell}$:
  \[  (\Gamma_{k,\ell})_b \ \ \ \in A^{2}_{hom}(  S_b\times S_b)\ \ \ \forall b\in B\  \ \   ( 1\le k,\ell\le m)\ .\]  
 This means that we can apply Voisin's key result, proposition \ref{propvois}, to the relative correspondence $\Gamma_{k,\ell}$. The conclusion is that for each $1\le k,\ell\le m$, there exists a cycle $\delta_{k,\ell}\in A^2(G\times G)$ (where $G=G(2,5)$ is the Grassmannian as in theorem \ref{muk}) such that
   \[  (\Gamma_{k,\ell})_b  +(\delta_{k,\ell})_b=0 \ \ \ \hbox{in}\ A^{2}_{}(  S_b\times S_b)\ \ \ \forall b\in B\  \ \ .\]    
   Since a Grassmannian has trivial Chow groups, this implies in particular that
   \[  \bigl(  (\Gamma_{k,\ell})_b \bigr){}_\ast =0\colon\ \ \ A^i_{hom}(S_b)\ \to\ A^i_{hom}(S_b)\ \ \ \forall b\in B\ .\]
   In view of equality (\ref{same2}), this implies
  \[  \bigl((\Gamma_1)_b\bigr){}_\ast \colon\ \ \ =0 \ \ \  A^i_{hom}\bigl((S_b)^m\bigr)  \ \to\  A^i_{hom}\bigl((S_b)^m\bigr)\ \ \ \forall b\in B\   ,\]    
  as claimed.
  
  (The argument for $\Gamma_1^\prime$ is the same; it suffices to replace the use of proposition \ref{prod3} by proposition \ref{prod2}.)
  Claim \ref{rat1} is now proven.
  
  It is high time to wrap up the proof of theorem \ref{main}. For $b\in B$ general, the restrictions $(\gamma_1)_b, (\gamma_1^\prime)_b$ of equation (\ref{modif}) will be supported on $D_b\times D_b\subset (S_b)^m\times (S_b)^m$, where $D_b\subset (S_b)^m$ is a divisor. As such, the action
    \[ \begin{split}   \bigl( (\gamma_1)_b\bigr){}_\ast\colon\ \ \ &R\bigl( (S_b)^m\bigr)\ \to\ R\bigl( (S_b)^m\bigr)\ ,\\       
                           \bigl( (\gamma^\prime_1)_b\bigr){}_\ast\colon\ \ \ &R\bigl( (S_b)^m\bigr)\ \to\ R\bigl( (S_b)^m\bigr)\ ,\\     
                       \end{split}\]
      is $0$ for general $b\in B$, where $R$ is either $A^2_{hom}$ or $A^{2m}$. Combining this observation with equation (\ref{modif}) and claim (\ref{rat1}), we find that
      \[  \begin{split}  (\Gamma_b )_\ast=0\colon\ \ \ R\bigl( (S_b)^m\bigr)\ \to\ R\bigl( (S_b)^m\bigr)\ ,\\
                               (\Gamma_b^\prime)_\ast=0\colon\ \ \ R\bigl( (S_b)^m\bigr)\ \to\ R\bigl( (S_b)^m\bigr)\ \\
                           \end{split}\]
                   (where, once more, $R$ is either $A^2_{hom}$ or $A^{2m}$).
                   
  In view of the definition (\ref{defg}) of $\Gamma, \Gamma^\prime$ (and using that the cycles $\gamma_b, \gamma^\prime_b$ occuring in (\ref{defg}) are supported in codimension $1$ for $b\in B$ general, and so act trivially on $A^2_{hom}()$ and on $A^{2m}()$), it follows that                
         \begin{equation}\label{this} \begin{split} \Bigl( \Pi_2^{(S_b)^m}\circ \Cc_b\circ \Pi_{4m-2}^{(S_b)^m}\circ (\Gamma_{L^{2m-2}})_b \circ \Pi_2^{(S_b)^m} -  \Pi_2^{(S_b)^m}\Bigr){}_\ast&=0\colon\ \ \ A^2_{hom}\bigl( (S_b)^m\bigr) \to      A^2_{hom}\bigl( (S_b)^m\bigr)\ ,\\
                        \Bigl( \Pi_{4m-2}^{(S_b)^m}\circ (\Gamma_{L^{2m-2}})_b\circ \Pi_2^{(S_b)^m}\circ \Cc_b\circ \Pi_{4m-2}^{(S_b)^m} - \Pi_{4m-2}^{(S_b)^m}\Bigr){}_\ast&=0\colon\ \ \ A^{2m}\bigl( (S_b)^m\bigr) \to A^{2m}\bigl( (S_b)^m\bigr)\ ,\\
                       \end{split}\end{equation}
                  for general $b\in B$. Since $\Pi_2^{(S_b)^m}$ acts as the identity on $A^2_{(2)}((S_b)^m)$, it follows from the first line of (\ref{this}) that
              \[       \Bigl( \Pi_2^{(S_b)^m}\circ \Cc_b\circ \Pi_{4m-2}^{(S_b)^m}\circ (\Gamma_{L^{2m-2}})_b\Bigr){}_\ast   =\ide\colon\ \ \ A^2_{(2)}\bigl( (S_b)^m\bigr)\ 
                  \to\ A^2_{(2)}\bigl( (S_b)^m\bigr)\ ;\]
             in particular
             \[ \cdot L^{2m-2}\colon\ \ \ A^2_{(2)}\bigl( (S_b)^m\bigr)\ \to\ A^{2m}_{(2)}\bigl( (S_b)^m\bigr) \]
             is injective for general $b\in B$. Likewise, it follows from the second line of (\ref{this}) that
 \[ \Bigl( \Pi_{4m-2}^{(S_b)^m}\circ (\Gamma_{L^{2m-2}})_b\circ \Pi_2^{(S_b)^m}\circ \Cc_b\Bigr){}_\ast =\ide\colon\ \ \ A^{2m}_{(2)}\bigl( (S_b)^m\bigr)\ \to\ A^{2m}_{(2)}\bigl( (S_b)^m\bigr)\]
  for general $b\in B$. However, the image of
    \[ A^2_{(2)}\bigl( (S_b)^m\bigr)\ \xrightarrow{\cdot L^{2m-2}}\ A^{2m}\bigl( (S_b)^m\bigr) \]
    is contained in $A^{2m}_{(2)}((S_b)^m)$, since $L\in A^1((S_b)^m)=A^1_{(0)}((S_b)^m)$, and so this further simplifies to
    \[ \Bigl( (\Gamma_{L^{2m-2}})_b\circ \Pi_2^{(S_b)^m}\circ \Cc_b\Bigr){}_\ast =\ide\colon\ \ \ A^{2m}_{(2)}\bigl( (S_b)^m\bigr)\ \to\ A^{2m}_{(2)}\bigl( (S_b)^m\bigr)\ \]
    for general $b\in B$. In particular,
      \[ \cdot L^{2m-2}\colon\ \ \ A^2_{(2)}\bigl( (S_b)^m\bigr)\ \to\ A^{2m}_{(2)}\bigl( (S_b)^m\bigr) \]
             is surjective for general $b\in B$.  
    
    Theorem \ref{hard} is now proven for general $b\in B$ (this suffices for the purposes of this note). To prove the theorem for {\em all\/} $b\in B$, one may observe that the above argument can be made to work ``locally around a given $b_0\in B$'', i.e. given $b_0\in B$ one can find relative correspondences $\gamma, \gamma^\prime, \ldots$ supported in codimension $1$ and in general position with respect to the fibre over $b_0$.     
                               
      \end{proof}

   Theorem \ref{hard} can be reformulated in terms of Hilbert schemes:       
 
 \begin{corollary}\label{hardhilb} Let $S_b$ be a $K3$ surface of degree $10$, and let $X=(S_b)^{[m]}$ be the Hilbert scheme of length $m$ subschemes of $S$. Let $L\in A^1(\Ss^{m/B})$ be a relatively big line bundle, and set
   \[ L_X := (f_b)^\ast (p_b)_\ast (L_b)\ \ \ \in A^1(X)\ ,\]
   where $p_b\colon (S^b)^m\to (S_b)^{(m)}$ denotes the projection, and $f_b\colon (S^b)^{[m]}\to (S_b)^{(m)}$ denotes the Hilbert--Chow morphism. Then
   \[  \cdot (L_X)^{m-1}\colon\ \ \ A^2_{(2)}(X)\ \to\ A^{2m}_{(2)}(X) \]
    is an isomorphism.
  \end{corollary}  
    
  \begin{proof} Let the symmetric group $\Sy_m$ act on $\Ss^{m/B}$ by permuting the factors, and let
   \[  p\colon\ \ \ \Ss^{m/B}\ \to\ \Ss^{m/B}/\Sy_m\]
   denote the quotient morphism. Theorem \ref{hard} applies to the line bundle
   \[ L^\prime:= p^\ast p_\ast (L) = \sum_{\sigma\in\Sy} \sigma^\ast(L)\ \ \ \in A^1(\Ss^{m/B})\ .\]
   There is a commutative diagram
     \[ \begin{array}[c]{ccc}
              A^2_{(2)}((S_b)^m)^{\Sy_m} & \xrightarrow{\cdot (L_b^\prime)^{m-1}}& A^{2m}_{(2)}((S_b)^m)^{\Sy_m}\\
              {\scriptstyle (p_b)^\ast} \uparrow{\scriptstyle \cong}  &&      {\scriptstyle (p_b)^\ast} \uparrow{\scriptstyle \cong} \\
               A^2_{(2)}((S_b)^{(m)}) & \xrightarrow{\cdot ((p_b)_\ast(L_b) )^{m-1}}& A^{2m}_{(2)}((S_b)^{(m)})\\   
             \end{array}\]   
       In view of theorem \ref{hard} (applied to $L^\prime$), the lower horizontal arrow is an isomorphism.      
                                   
  It follows from the de Cataldo--Migliorini isomorphism of motives \cite{CM} that there is a correspondence--induced isomorphism
    \[ A^2(X)\cong A^2((S_b)^{(3)})\oplus A^1()\oplus A^0()\ ,\]
    and so in particular an isomorphism
    \[ A^2_{AJ}(X)\cong A^2_{AJ}((S_b)^{(3)})  \ .\]
    Since $A^2_{(2)}()\subset A^2_{AJ}()$, and the de Cataldo--Migliorini respects the bigrading (by construction of the latter), this implies that
     \[ f^\ast\colon\ \ \ A^2_{(2)}((S_b)^{(m)})\ \to\ A^2_{(2)}(X) \]
     is an isomorphism.
     
     Similarly, there is an isomorphism
     \[ f^\ast\colon\ \ \ A^{2m}((S_b)^{(m)})\ \xrightarrow{\cong}\ A^{2m}(X) \]
     which respects the bigrading. 
     
     Corollary \ref{hardhilb} now follows from what we have said above, in view of the commutative diagram 
     \[ \begin{array}[c]{ccc}
              A^2_{(2)}(X) & \xrightarrow{\cdot (L_X)^{m-1}}& A^{2m}_{(2)}(X)\\
              {\scriptstyle (f_b)^\ast} \uparrow{\scriptstyle \cong}  &&      {\scriptstyle (f_b)^\ast} \uparrow{\scriptstyle \cong} \\
               A^2_{(2)}((S_b)^{(m)}) & \xrightarrow{\cdot (p_\ast(L_b) )^{m-1}}& A^{2m}_{(2)}((S_b)^{(m)})\\   
             \end{array}\]                 
  \end{proof}

   One can also reformulate theorem \ref{hard} in terms of double EPW cubes; this will come in useful when proving our main result (theorem \ref{main}). 
   
   \begin{corollary}\label{hardepw} Let $\XX_{\Delta^1}\to \Delta^1$ be the family of double EPW cubes parametrized by the divisor $\Delta^1\subset LG^1_\nu$ of theorem \ref{ikkr}. Let $L\in A^1(\XX_{\Delta^1})$ be a line bundle that is in the image of the pullback map
   \[   A^1(\EE)\ \xrightarrow{h^\ast}\ A^1(\XX_{\Delta^1}) \]
   (where $h\colon\XX_{\Delta^1}\dashrightarrow \EE$ is as in proposition \ref{families}). Assume $L$ is relatively big.
      Then
     \[   \cdot (L_A)^4\colon\ \ \ A^2_{(2)}(X_A)\ \to\ A^6_{(2)}(X_A) \]
    is an isomorphism for general $A\in\Delta^1$. 
  \end{corollary} 
  
 \begin{proof} Let us write $L=h^\ast(L_\EE)$, where $L_\EE\in A^1(\EE)$ is relatively big.
  
   Let $\Ss^{3/B^0}$ denote the family of third powers of degree $10$ $K3$ surfaces over the Zariski open $B^0\subset B$ as in proposition \ref{families}. We have seen (theorem \ref{ikkr}) that for a general $A\in\Delta^1$ there is $b\in B$ such that $A=f(b)$ and there is a birational map
             \[   (S_b)^{[3]}\ \stackrel{\phi_b}{\dashrightarrow}\ X_{A}\ .\]  
       This fits into a commutative diagram
       \[  \begin{array}[c]{ccccc}
            (S_b)^3 & \stackrel{\Phi_b}{\dashrightarrow}&\ \ \ \  (S_b)^{[3]}& \stackrel{\phi_b}{\dashrightarrow}& X_A\\
             & \searrow\ \ \ \ \swarrow && \ \ \ \searrow{\scriptstyle \Psi_b} \ \ \ \  \ \ \ \swarrow{\scriptstyle h_b}   &\\
              & (S_b)^{(3)} &\ \ \ \stackrel{\Psi_b^\prime}{\dashrightarrow}& D_A&\\
              \end{array}\]

                The pullback $L_{\Ss}:=\Phi^\ast \Psi^\ast (L_\EE)\in A^1(\Ss^{3/B^0})$ is relatively big, and so theorem \ref{hard} applies to $L_{\Ss}$. 
                There is a commutative diagram
             \[  \begin{array}[c]{ccc}
                   A^2_{(2)}\bigl( (S_b)^3\bigr)^{\Sy_3} & \xrightarrow { \cdot  ( (L_{\Ss})_b)^4}&  A^6_{(2)}\bigl( (S_b)^3\bigr)^{\Sy_3}\\
                  \ \ \ \  \uparrow {\scriptstyle \cong} &&   \ \ \ \  \uparrow {\scriptstyle \cong} \\   
             A^2_{(2)}\bigl( (S_b)^{(3)}\bigr)^{} & \xrightarrow { \cdot  ( (\Psi^\prime_b)^\ast((L_\EE)_A))^4}&  A^6_{(2)}\bigl( (S_b)^{(3)}\bigr)^{}\\
                \ \ \ \  \downarrow {\scriptstyle \cong} &&   \ \ \ \  \downarrow {\scriptstyle \cong}  \\  
     A^2_{(2)}\bigl( (S_b)^{[3]}\bigr)^{} & \xrightarrow { \cdot  ( (\Psi_b)^\ast((L_\EE)_A))^4}&  A^6_{(2)}\bigl( (S_b)^{[3]}\bigr)^{}\\
              \ \ \ \  \uparrow {\scriptstyle \cong} &&   \ \ \ \  \uparrow {\scriptstyle \cong} \\   
           A^2_{(2)}(X_{A}) & \xrightarrow{   \cdot (L_{A})^4} & A^6_{(2)}(X_{A})\\
           \end{array} \]
           (Here the lowest vertical arrows are isomorphisms thanks to Rie\ss's isomorphism \cite{Rie}. The lowest square is commutative, because $\phi_b$ is a codimension $1$ isomorphism, and the divisors $L_A=(h_b)^\ast((L_\EE)_A))$ and $(\Psi_b)^\ast((L_\EE)_A)$ coincide on the open where $\phi_b$ is an isomorphism.) Theorem \ref{hard} implies the top horizontal arrow is an isomorphism. It follows that all horizontal arrows are isomorphisms, and corollary \ref{hardepw} is proven.
         \end{proof}

 \begin{remark}\label{question} Looking at corollary \ref{hardhilb}, one might hope that a similar result is true more generally. 
 
 Let $X$ be any hyperk\"ahler variety of dimension $2m$, and suppose the Chow ring of $X$ has a bigraded ring structure $A^\ast_{(\ast)}(X)$. One can ask the following questions:
 
 \noindent
 (\rom1) Let $L\in A^1(X)$ be an ample line bundle. Is it true that there are isomorphisms
   \[  \cdot L^{2m-2i+j}\colon\ \  A^i_{(j)}(X)\ \xrightarrow{\cong}\ A^{2m-i+j}_{(j)}(X)\ \ \ \hbox{for\ all\ } 0\le 2i-j\le 2m\  \ ?\]   
   
 \noindent
 (\rom2) Let $L\in A^1(X)$ be a big line bundle. Is it true that there are isomorphisms
  \[    \cdot L^{2m-i}\colon\ \  A^i_{(i)}(X)\ \xrightarrow{\cong}\ A^{2m}_{(i)}(X)\ \ \ \hbox{for\ all\ } 0\le i\le 2m\ \ ?\]     
  
  The answer to the first question is ``yes'' for generalized Kummer varieties \cite{hard}.
  The answer to both questions is ``I don't know'' for Hilbert schemes of $K3$ surfaces. 
  
  (However, if the $K3$ surface $S$ has small genus there exists a Mukai model, and presumably the above proof can then be extended to settle questions (\rom1) and (\rom2) affirmatively for $A^2_{(2)}(S^{[m]})$ and line bundles $L$ that exist relatively. The question for $A^i_{ (j)}(S^{[m]})$ with $i>2$ becomes more complicated, as one would need an analogon of proposition \ref{propvois} for higher fibre products $\Ss^{m/B}$ with $m>2$.)
   \end{remark}   
    
  \begin{remark} Let $X$ be either $S^m$ or $S^{[m]}$ where $S$ is a degree $10$ $K3$ surface. It follows from (the proof of) corollary \ref{hardhilb} that 
    \[A^2_{(2)}(X)\ \subset\ A^2_{alg}(X)\ ,\]
    where $A^\ast_{alg}()\subset A^\ast()$ denotes the subgroup of algebraically trivial cycles. This is in agreement with a conjecture of Jannsen \cite{J3}, stipulating that for any smooth projective variety $Z$ one should have
    \[ F^i A^i(Z)\ \subset\ A^i_{alg}(Z)\ ,\]
    where $F^\ast$ is the conjectural Bloch--Beilinson filtration.
  \end{remark} 
    
  \begin{remark} Let $X$ be either $S^m$ or $S^{[m]}$ where $S$ is a degree $10$ $K3$ surface, and let $L\in A^1(X)$ be a line bundle as in theorem \ref{hard} (resp. as in corollary \ref{hardhilb}).
 Provided $L$ is sufficiently ample, there exists a smooth complete intersection surface $Y\subset X$ defined by the linear system $\vert L\vert$. Theorem \ref{hard} (resp. corollary \ref{hardhilb}) then implies that $A^{2m}_{(2)}(X)$ is supported on $Y$, and that
   \[  A^2_{(2)}(X)\ \to\ A^2(Y) \]
   is injective. This injectivity statement is in agreement with Hartshorne's ``weak Lefschetz'' conjecture for Chow groups \cite{Ha} (we recall that it is expected that
   $A^2_{(2)}(X)=A^2_{hom}(X)$ for these $X$).
    \end{remark}

 \section{Main result}

 \begin{theorem}\label{main} Let $X$ be a double EPW cube, and assume $X=X_A$ for $A\in \Delta^1$ general (where $\Delta^1\subset LG^1_\nu$ is the divisor of theorem \ref{ikkr}). Let $\iota=\iota_A\in\aut(X)$ be the anti--symplectic involution given by the double cover $X_A\to D_2^A$. Then
  \[ \begin{split}  \iota^\ast&=-\ide\colon\ \ \ A^6_{(2)}(X)\ \to\ A^6(X)\ ,\\
                        (\Pi_2^X)_\ast \iota^\ast&=-\ide\colon\ \ \ A^2_{(2)}(X)\ \to\ A^2_{(2)}(X)\ .
                        \end{split}\]   
      \end{theorem}

 \begin{proof} In a first reduction step, we show that it suffices to prove the first statement of theorem \ref{main}. Let 
   \[ \XX_{\Delta^{1,0}}\ \to\ \DD_{\Delta^{1,0}} \ \to\ \Delta^{1,0}\]
   be the families as in theorem \ref{ikkr}, so a fibre $D_A$ of $\DD_{\Delta^{1,0}}$ over $A\in \Delta^{1,0}$ is an EPW cube, and a fibre $X_A$ of $\XX_{\Delta^{1,0}}$ over $A$ is a double EPW cube birational to a Hilbert cube $K3^{[3]}$.
   Taking the restriction of a $P$--invariant ample line bundle on the Grassmannian, one can find a relatively ample line bundle $L_\EE\in A^1(\EE)=A^1_P(\EE)$, where $\EE=\DD_{\Delta^{1,0}}/P$ is as in proposition \ref{families}. Pulling back to $\XX_{\Delta^{1,0}}$, one obtains a $\iota$--invariant relatively ample line bundle in $ A^1(\XX_{\Delta^{1,0}})$. 
  
   Applying corollary \ref{hardepw} to $X=X_A$ for $A\in\Delta^1$ general, one obtains an isomorphism
   \begin{equation}\label{aniso}  \cdot (L\vert_X)^{4}\colon\ \ \ A^2_{(2)}(X)\ \xrightarrow{\cong}\ A^{6}_{(2)}(X)\ . \end{equation}
   But $L\vert_X$ is $\iota$--invariant by construction, and so 
   \[ \iota^\ast \bigl(  (L\vert_X)^4\cdot b)\bigr) = (L\vert_X)^4\cdot \iota^\ast(b)\ \ \ \hbox{in}\  A^6(X)\ \ \ \forall b\in A^2_{(2)}(X)\ .\]
    Suppose now the first statement of theorem \ref{main} holds true. Then we find that
    \[  (L\vert_X)^4\cdot \bigl(b+\iota^\ast(b)\bigr)=0 \ \ \ \hbox{in}\  A^6(X)\ \ \ \forall b\in A^2_{(2)}(X)\ .\]
    In view of the isomorphism (\ref{aniso}), this implies
    \[ \iota^\ast(b)=-b +b_0\ \ \ \hbox{in}\ A^2(X)\ ,\]
    where $b_0\in A^2_{(0)}(X)$ (and actually $b_0\in A^2_{(0),hom}(X)$, which is conjecturally $0$). This proves the second statement of theorem \ref{main}. It remains to prove the first statement of theorem \ref{main}.
    
        In view of Rie\ss's isomorphism, to prove the first statement it suffices to prove that
   \begin{equation}\label{onhilb} (\phi_b)^\ast (\iota_b)^\ast (\phi_b)_\ast=-\ide\colon\ \ \ A^6_{(2)}\bigl( (S_b)^{[3]}\bigr)\ \to\ A^6_{}\bigl( (S_b)^{[3]}\bigr)   \ ,\end{equation}
   where $S_b$ is a general degree $10$ $K3$ surface and $\phi_b\colon (S_b)^{[3]}\dashrightarrow X$ is the birational map.   
   
   Consider now the commutative square
    \[ \begin{array}[c]{ccc}
        (S_b)^{[3]} &\leftarrow &\wt{(S_b)^{3}}\\
        \downarrow&&\downarrow\\
        (S_b)^{(3)} & \leftarrow & (S_b)^3\\
        \end{array}\]
        (where vertical arrows are a composition of blow--ups of various partial diagonals). This gives rise to a correspondence $\Phi_b\in A^6((S_b)^{[3]}\times (S_b)^3)$, and the blow--up exact sequence implies that
        \[ (\Phi_b)^\ast (\Phi_b)_\ast=\ide\colon\ \ \ A^6_{}\bigl( (S_b)^{[3]}\bigr)\ \to\ A^6_{}\bigl( (S_b)^{[3]}\bigr)  \ .\]
        
    Therefore, we can work with the self--product $(S_b)^3$ rather than the Hilbert cube $(S_b)^{[3]}$: to prove (\ref{onhilb}), it suffices to prove that    
     \begin{equation}\label{oncube0}  (\Phi_b)^\ast (\phi_b)^\ast \iota^\ast (\phi_b)_\ast (\Phi_b)_\ast=-\ide\colon\ \ \ A^6_{(2)}\bigl( (S_b)^{3}\bigr)\ \to\ A^6_{(2)}\bigl( (S_b)^{3}\bigr)   \ ,\end{equation} 
         for general $b\in B$.    
         
      Thanks to the following compatibility lemma, things further simplify:
      
           \begin{lemma}\label{compati} Let $\TT^{[3]/\MM^0_6 }\to \MM^0_6$ be the ``universal Hilbert cube'' as in proposition \ref{families}, and let
        \[ \iota_T\colon\ \ \ \TT^{[3]/\MM^0_6 }\ \dashrightarrow\        \TT^{[3]/\MM^0_6 } \]
        be the birational involution induced by the generically $2:1$ rational map $\Psi\colon \TT^{[3]/\MM^0_6 }\dashrightarrow \EE$ of proposition \ref{families}.
        Let $\Gamma_{\iota_S}$ be the relative correspondence
          \[ \Gamma_{\iota_S}:=  {}^t \bar{\Gamma}_g\circ \bar{\Gamma}_{\iota_T}\circ \bar{\Gamma}_g\ \ \ \in\ A^6( \Ss^{3/B^0}\times_{B^0} \Ss^{3/B^0}) \ \]
          (where $g\colon \Ss^{3/B^0}\dashrightarrow \TT^{[3]/\MM^0_6}$ is the natural rational map).
          
        Then there is equality
         \[  \bigl((\Gamma_{\iota_S})_b\bigr){}_\ast = (\Phi_b)^\ast(\phi_b)^\ast \iota^\ast (\phi_b)_\ast(\Phi_b)_\ast\colon\ \ \ A^6\bigl( (S_b)^{3}\bigr)\ \to\ A^6_{}\bigl( (S_b)^{3}\bigr)   \ \]
         for general $b\in B$.
                \end{lemma}
                
             \begin{proof} One should remember that for general $b\in B$, there is a birational map
               \[ \phi_b\colon\ \ \ (S_b)^{[3]}\ \dashrightarrow\ X:=X_{f(b)}\ ,\]
               where $f(b)\in \MM_{\Delta^0}$ in the notation of proposition \ref{families}.
               Let
               \[ \begin{array}[c]{ccccc}
                 && Z &&\\
                 &{\scriptstyle p}\swarrow\ \ \ && \ \ \ \searrow{\scriptstyle q}\\
                S_b&& \stackrel{\phi_b}{\dashrightarrow}&& X\\
                \end{array}\]
                be an elimination of indeterminacy. 
                Let $\iota_Z\colon Z\to Z$ be the birational involution induced by $\iota$. There is a commutative diagram
            \begin{equation}\label{anothercom}\begin{array} [c]{ccccc}
                   (S_b)^{[3]} & \stackrel{p}{\leftarrow}& Z& \stackrel{q}{\to}& X\\
                    \ \ \  \downarrow{\scriptstyle \iota_{S_b}}&&\ \ \  \downarrow{\scriptstyle \iota_Z}&&\ \ \ \downarrow{\scriptstyle \iota}\\
                     (S_b)^{[3]} & \stackrel{p}{\leftarrow}& Z& \stackrel{q}{\to}& X\\
                     \end{array}\end{equation}
                     (here $\iota_{S_b}$ and $\iota_Z$ are birational maps, not morphisms).
                     
       For general $b\in B$, the restriction $(\bar{\Gamma}_{\iota_T})_b$ is just the closure of the graph of the rational involution $\iota_{S_b}\colon S_b\dashrightarrow S_b$ (induced by $\iota$), and so
             \begin{equation}\label{lh}    \bigl((\Gamma_{\iota_S})_b\bigr){}_\ast = (\Phi_b)^\ast(\iota_{S_b})^\ast (\Phi_b)_\ast =(\Phi_b)^\ast p_\ast (\iota_Z)^\ast p^\ast (\Phi_b)_\ast \colon\ \ \ 
             A^i\bigl((S_b)^3\bigr)\ \to\ 
                     A^i\bigl((S_b)^3\bigr)\ .\end{equation}
        As for the right--hand--side in lemma \ref{compati}, since $\iota^\ast= q_\ast (\iota_Z)^\ast q^\ast$ and
        $(\phi_b)^\ast=p_\ast q^\ast$ (and likewise $(\phi_b)_\ast= q_\ast p^\ast$), we find that
        \[ (\Phi_b)^\ast(\phi_b)^\ast \iota^\ast (\phi_b)_\ast(\Phi_b)_\ast= (\Phi_b)^\ast p_\ast q^\ast q_\ast(\iota_Z)^\ast q^\ast q_\ast p^\ast(\Phi_b)_\ast
                \colon\ \ \ A^i\bigl( (S_b)^{3}\bigr)\ \to\ A^i_{}\bigl( (S_b)^{3}\bigr)\ .   \ \]
         Since $q\colon Z\to X$ is birational, we have that $q^\ast q_\ast=\ide\colon A^6(Z)\to A^6(Z)$, and so for $i=6$ the above boils down to
           \begin{equation}\label{rh} (\Phi_b)^\ast(\phi_b)^\ast \iota^\ast (\phi_b)_\ast(\Phi_b)_\ast= (\Phi_b)^\ast p_\ast (\iota_Z)^\ast  p^\ast(\Phi_b)_\ast
                \colon\ \ \ A^6\bigl( (S_b)^{3}\bigr)\ \to\ A^6_{}\bigl( (S_b)^{3}\bigr)\ .   \ \end{equation}      
                Comparing equations (\ref{lh}) and (\ref{rh}), we ascertain that we have proven the lemma.
     \end{proof}   
             
        Thanks to lemma \ref{compati}, we conclude that in order to prove (\ref{oncube0}), it suffices to prove that
        \begin{equation}\label{oncube1}       \bigl((\Gamma_{\iota_S})_b\bigr){}_\ast = -\ide\colon\ \ \ A^6_{(2)}\bigl( (S_b)^{3}\bigr)\ \to\ A^6_{}\bigl( (S_b)^{3}\bigr)   \ ,\end{equation} 
         for general $b\in B$.

         We now introduce one further reduction step: we claim that in order to prove statement (\ref{oncube1}), it suffices to prove that
         \begin{equation}\label{oncube} (\Pi_{10}^{(S_b)^3})_\ast  \bigl((\Gamma_{\iota_S})_b\bigr){}_\ast  =-\ide\colon\ \ \ A^6_{(2)}\bigl( (S_b)^{3}\bigr)\ \to\ A^6_{(2)}\bigl( (S_b)^{3}\bigr)   \ ,\end{equation} 
         for general $b\in B$.
         
   To prove this claim, we observe that equation (\ref{oncube}) implies (by composing on both sides) that      
   \[  (\phi_b)_\ast  (\Phi_b)_\ast (\Pi_{10}^{(S_b)^3})_\ast     \bigl((\Gamma_{\iota_S})_b\bigr){}_\ast  (\Phi_b)^\ast (\phi_b)^\ast    = -\ide\colon\ \ \ 
        A^6(X)\ \to\ A^6(X)\ ,\]
        for general $b\in B$. Here $X=X_A$ is the double EPW cube such that 
          \[\phi_b\colon\ \ \  (S_b)^{[3]}\ \dashrightarrow\ X\] 
          is birational.
        Using lemma \ref{compati}, this implies that also
     \[ (\phi_b)_\ast  (\Phi_b)_\ast (\Pi_{10}^{(S_b)^3})_\ast  (\Phi_b)^\ast (\phi_b)^\ast \iota^\ast (\phi_b)_\ast (\Phi_b)_\ast (\Phi_b)^\ast (\phi_b)^\ast   =-\ide\colon\ \ \  A^6(X)\ \to\ A^6(X)\ ,\]
      for general $X=X_A$ with $A\in\Delta^1$.
       This simplifies to
        \begin{equation}\label{step} (\phi_b)_\ast  (\Phi_b)_\ast (\Pi_{10}^{(S_b)^3})_\ast  (\Phi_b)^\ast (\phi_b)^\ast \iota^\ast   =-\ide\colon\ \ \       
        A^6(X)\ \to\ A^6(X)\ .\end{equation}
     But 
       \[     (\Phi_b)_\ast (\Pi_{10}^{(S_b)^3})_\ast  =(\Pi_{10}^{(S_b)^{[3]}})_\ast (\Phi_b)_\ast \colon\ \ \ A^6((S_b)^3)\ \to\ A^6((S_b)^{[3]}) \]
       (lemma \ref{compat} ),
       and 
       \[  (\phi_b)_\ast           (\Pi_{10}^{(S_b)^{[3]}})_\ast  = (\Pi_{10}^X)_\ast (\phi_b)_\ast\colon\ \ \ A^i((S_b)^{[3]})\ \to\ A^i(X) \]
       (since Rie\ss's isomorphism is an isomorphism of bigraded rings, cf. lemma \ref{hk}). Therefore, equation (\ref{step}) further simplifies to
       \[   (\Pi_{10}^X)_\ast \iota^\ast =-\ide\colon\ \ \ A^6_{(2)}(X)\ \to\ A^6_{(2)}(X)\ .\]
       This means that any $b\in  A^6_{(2)}(X)$ satisfies
       \begin{equation}\label{46} \iota^\ast(b) = -b +b_4 +b_6\ \ \ \hbox{in}\ A^6(X)\ ,\end{equation}
       where $b_j\in A^6_{(j)}(X)$ (NB: $\iota^\ast(b)$ cannot have a component in $A^6_{(0)}(X)$ since $\iota^\ast(b)\in A^6_{hom}(X)$.)
       On the other hand, using corollary \ref{hardepw} (just as at the beginning of this proof) we can write $b=L^4\cdot a$ where $a\in A^2_{(2)}(X)$ and $L$ is a $\iota$--invariant ample line bundle. This implies that
       \[ \iota^\ast(b)=\iota^\ast(L^4\cdot a)=\iota^\ast(L^4)\cdot \iota^\ast(a)= L^4\cdot \iota^\ast(a)\ \ \ \hbox{in}\ A^6(X)\ .\]
       But $\iota^\ast(a)\in A^2(X)=A^2_{(0)}(X)\oplus A^2_{(2)}(X)$ and so (exploiting the fact that $A^\ast_{(\ast)}(X)$ is a bigraded ring, thanks to lemma \ref{mck}) we find that
       \begin{equation}\label{02} \iota^\ast(b)\ \ \in A^6_{(0)}(X)\oplus A^6_{(2)}(X)\ .\end{equation}
       Comparing equations (\ref{02}) and (\ref{46}), we see that we must have $b_4=b_6=0$, and so
       \[ \iota^\ast(b)=-b\ \ \ \hbox{in}\ A^6(X)\ \ \ \forall\ b\in A^6_{(2)}(X)\ ,\]
       as claimed. This proves the claim; it now remains to prove statement (\ref{oncube}).

                In order to prove statement (\ref{oncube}), we rely once again on the machinery of ``spread'' of cycles in a family \cite{V0}, \cite{V1}; this is very similar to the argument proving theorem \ref{hard}.
  We consider the family
   \[ \Ss^{3/B}\ \to\ B\ ,\]
   where $\Ss\to B$ is (once more) the universal family of degree $10$ $K3$ surfaces (remark \ref{family}).

     Let us define a relative correspondence
     \[ \Gamma:=   \Pi_{10}^{\Ss^{3/B}}\circ \bigl(   \Gamma_{\iota_S}  +\Delta_{\Ss^{3/B}}   \bigr) 
         \circ\Pi_{10}^{\Ss^{3/B}} \ \ \ \in\ A^6\bigl( \Ss^{3/B}\times_B \Ss^{3/B}\bigr)\ .\]
      Clearly, statement (\ref{oncube}) that we want to prove is equivalent to the statement
     \begin{equation}\label{noaction}    \bigl( \Gamma_b\bigr){}_\ast=0\colon\ \ \ A^6_{hom}\bigl( (S_b)^{3}\bigr)\ \to\ A^6_{}\bigl( (S_b)^{3}\bigr)
         \ \ \ \hbox{for\ general\ $b\in B$}\ .\end{equation}
        (Here, as before, for any relative correspondence $\Gamma$ we use the notation $\Gamma_b$ to indicate the restriction of $\Gamma$ to the fibre over 
        $b\in B$.)
                   
  The homological input that we have at our disposition is that the involution $\iota=\iota_A$ of $X=X_A$ (and hence the induced involution of $(S_b)^{[3]}$) is anti--symplectic (remark \ref{anti}), and so
   \[   \bigl( \Gamma_b\bigr){}_\ast=0\colon\ \ \ H^{6,4}\bigl( (S_b)^3\bigr)\ \to\ H^{6,4}\bigl( (S_b)^3\bigr)\ \ \  \hbox{for\ general\ $b\in B$}\ .\]    
   Using the Lefschetz $(1,1)$--theorem, this implies that for general $b\in B$, there exist a curve $V_b$ and a divisor $W_b$ inside $(S_b)^3$, and a cycle $\gamma_b\in A_6(W_b\times V_b)$ such that
    \[ \Gamma_b +\gamma_b=0\ \ \ \hbox{in}\ H^{12}\bigl( (S_b)^3\times (S_b)^3\bigr)\ .\]
    Applying the Hilbert schemes argument \cite[Proposition 3.7]{V0}, one can find a curve $\VV$ and a divisor $\WW$ inside $\Ss^{3/B}$, and a cycle $\gamma$ supported on $\WW\times_B \VV$ such that
    \begin{equation}\label{no} \bigl(  \Gamma + \gamma\bigr){}_b=0\ \ \ \hbox{in}\ H^{12}\bigl( (S_b)^3\times (S_b)^3\bigr)\ \ \ \forall b\in B\ .\end{equation}  
    
     Let us now consider a modified relative correspondence
     \[ \Gamma_1:= \Pi_{10}^{\Ss^{3/B}}\circ (\Gamma+\gamma)\circ \Pi_{10}^{\Ss^{3/B}}\ \ \ \in\ A^6\bigl( \Ss^{3/B}\times_B \Ss^{3/B}\bigr)\ .\]
     Since $\Pi_{10}^{(S_b)^3}$ is idempotent for all $b\in B$, there is a fibrewise equality
     \[  (\Gamma_1)_b = (\Gamma+\gamma^\prime)_b\  \ \ \hbox{in}\ A^6\bigl( (S_b)^3\times (S_b)^3\bigr)\ \ \ \forall b\in B\ ,\]
     where $\gamma^\prime$ is (just like $\gamma$) a cycle supported on $\WW\times_B \VV$.
     For a general fibre, the restriction $(\gamma^\prime)_b$ will be supported on (divisor)$\times$(curve) and as such will not act on $A^6((S_b)^3)$.  It follows that
     \begin{equation}\label{same1}   \bigl( (\Gamma_1)_b\bigr){}_\ast = (\Gamma_b)_\ast\colon\ \ \ A^6\bigl( (S_b)^3\bigr)\ \to\  A^6\bigl( (S_b)^3\bigr)\ \ \ \hbox{for\ $b\in B$\ general}\ .\end{equation}     
     On the other hand, in view of proposition \ref{prod2}, there is a fibrewise equality of action
      \[ \bigl((\Gamma_1)_b\bigr){}_\ast = 
       \Bigl( \bigl( (\sum_{k=1}^3 \Xi_k\circ\Theta_k) \circ (\Gamma+\gamma)\circ       (\sum_{\ell=1}^3 \Xi_\ell\circ\Theta_\ell)  \bigr){}_b\Bigr){}_\ast\colon\ \ \ 
       A^6\bigl( (S_b)^3\bigr)\ \to\  A^6\bigl( (S_b)^3\bigr) \ \ \ \forall b\in B\  .\]
      That is, we have equality
      \begin{equation}\label{10}    \bigl( (\Gamma_1)_b \bigr){}_\ast= \Bigl( \bigl( \sum_{k=1}^3 \sum_{\ell=1}^3  \Xi_k\circ \Gamma_{k,\ell}\circ \Theta_\ell\bigr){}_b\Bigr){}_\ast\colon\ \ \ A^6\bigl( (S_b)^3\bigr)\ \to\ A^6\bigl( (S_b)^3\bigr)\ \ \ \forall b\in B ,\end{equation}
      where we have defined
       \[ \Gamma_{k,\ell}:= \Theta_k \circ(\Gamma+\gamma)\circ \Xi_\ell\ \ \ \in A^2(\Ss\times_B \Ss)\ \ \ 1\le k,\ell\le 3\ .\]
       
    We observe that equation (\ref{no}) implies that
    \[  (\Gamma_{k,\ell})_b\ \ \ \in A^2_{hom}(S_b\times S_b)\ \ \ \forall b\in B\ , 1\le k,\ell\le 3\ .\]
    But then, applying proposition \ref{propvois} to the relative correspondence $\Gamma_{k,\ell}$ we may conclude there exists $\delta_{k,\ell}\in A^2(G\times G)$ (where $G$ is the Grassmannian of lines in $\PP^4$) such that
    \[      (\Gamma_{k,\ell})_b +(\delta_{k,\ell})_b=0\ \ \ \in A^2_{}(S_b\times S_b)\ \ \ \forall b\in B  \ .\]
    Since the Grassmannian has trivial Chow groups, the correspondence $(\delta_{k,\ell})_b$ acts trivially on $A^\ast_{hom}(S_b)$, and so
    \[         \bigl( (\Gamma_{k,\ell})_b     \bigr){}_\ast =0\colon\ \ \ A^\ast_{hom}(S_b)\ \to\ A^\ast(S_b)\ \ \ \forall b\in B\ ,\ \ \ 1\le k,\ell\le 3\ .\]
    Plugging this in equation (\ref{10}), we find that
    \[ \bigl( (\Gamma_1)_b\bigr){}_\ast=0\colon\ \ \ A^6_{hom}\bigl( (S_b)^3\bigr)\ \to\  A^6_{}\bigl( (S_b)^3\bigr)  \ \ \ \forall b\in B\ .\]  
       Returning to equality (\ref{same1}), this implies that
    \[   (  \Gamma_b)_\ast=0\colon\ \ \ A^6_{hom}\bigl( (S_b)^3\bigr)\ \to\ A^6_{}\bigl( (S_b)^3\bigr)\ \ \ \hbox{for\ general\ $b\in B$}\ ,\]
    which is exactly statement (\ref{noaction}) that we needed to prove. The proof of theorem \ref{main} is now complete.
     \end{proof}

 \section{Some corollaries}
 
 \begin{corollary}\label{cor}
 Let $D=D_2^A$ be an EPW cube for $A\in\Delta^1$ general (where $\Delta^1\subset LG^1_\nu$ is the divisor of theorem \ref{ikkr}). 
  
  \noindent
  (\rom1) Let $a\in A^6(D)$ be a $0$--cycle which is either in the image of the intersection product map
    \[ A^2(D)\otimes A^2(D)\otimes A^2(D)\ \to\ A^6(D)\ ,\] 
    or in the image of the intersection product map
    \[ A^3(D)\otimes A^2(D)\otimes A^1(D)\ \to\ A^6(D)\ .\]
    Then $a$ is rationally trivial if and only if $a$ has degree $0$.
    
   \noindent
   (\rom2) Let $a\in A^5(D)$ be a $1$--cycle which is in the image of the intersection product map
    \[ A^2(D)\otimes A^2(D)\otimes A^1(D)\ \to\ A^5(D)\ .\]
    Then $a$ is rationally trivial if and only if $a$ is homologically trivial. 
  \end{corollary}   
 
 \begin{proof} We first establish some lemmas:
 
 \begin{lemma}\label{60} Let $A\in\Delta^1$ be general, and let $X=X_A$ be the corresponding double EPW cube. Let $\iota=\iota_A$ be the covering involution.
 Then
   \[ \iota^\ast=\ide\colon\ \ \ A^6_{(0)}(X)\ \to\ A^6(X)\ .\]
   \end{lemma}
   
   \begin{proof} The subgroup $A^6_{(0)}(X)$ is generated by $L^6$, where $L$ is any ample divisor. Taking $L$ an ample divisor of the form $L=p^\ast(L_D)$ where $L_D$ is ample on $D$, we see that the lemma must be true.
   \end{proof}

 \begin{lemma}\label{in0} Let $A\in \Delta^1$ be general, and let $X=X_A$ and $D=D_A$ be the corresponding double EPW cube, resp. EPW cube.
 Let $p\colon X\to D$ be the quotient morphism. We have
    \[ p^\ast A^2(D)\ \subset\ A^2_{(0)}(X)\ .\]
   \end{lemma}
   
   \begin{proof} By construction, there is an inclusion
         \[ p^\ast A^2(D)\ \subset\ A^2(X)^{\iota}\ ,\]
         where $\iota=\iota_A\in\aut(X)$ is the covering involution.
    
    Given $b\in A^2(D)$, let us write
     \[ p^\ast(b) = c_0 + c_2\ \ \ \in A^2_{(0)}(X^\prime)\oplus A^2_{(2)}(X^\prime)\ .\]
     Applying $\iota$, we find
       \begin{equation}\label{thishere}   \iota^\ast p^\ast(b) = p^\ast(b)=c_0 + c_2 \ \ \ \in A^2_{(0)}(X^\prime)\oplus A^2_{(2)}(X^\prime)\ .\end{equation}
      On the other hand, we have       
            \begin{equation}\label{thistoo} \iota^\ast p^\ast(b) = \iota^\ast(c_0) +\iota^\ast(c_2) =\iota^\ast(c_0) +d_0 - c_2\ \ \ \in A^2_{(0)}(X)\oplus A^2_{(2)}(X)\ ,
            \ \end{equation}
     where we have used sublemma \ref{preserve} below to obtain that $\iota^\ast(c_0)\in A^2_{(0)}(X)$, and theorem \ref{main} to obtain that $\iota^\ast(c_2)=-c_2+d_0$ for some $d_0\in A^2_{(0)}(X)$.  
     Comparing expressions (\ref{thishere}) and (\ref{thistoo}), we find
      \[  \iota^\ast(c_0) +d_0= c_0\ \ \ \hbox{in}\ A^2_{(0)}(X)\ ,\ \ \ -c_2=c_2\ \ \ \hbox{in}\ A^2_{(2)}(X)\ ,\]
      proving lemma \ref{in0}.

      \begin{sublemma}\label{preserve} Set--up as above. 
      Let $b\in A^2(D)$, and write
       \[ p^\ast(b) = c_0 + c_2\ \ \ \in A^2_{(0)}(X)\oplus A^2_{(2)}(X)\ .\]   
       Then
            \[ \iota^\ast (c_0)\in\  A^2_{(0)}(X)\ .\]
  \end{sublemma}
  
  \begin{proof} 
   Suppose
   \[ \iota^\ast(c_0) = d_0 + d_2\ \ \ \hbox{in}\ A^2(X)\ ,\]  
   with $d_0\in A^2_{(0)}(X)$ and $d_2\in A^2_{(2)}(X)$.
  
  Let $L\in A^1(X)$ be a $\iota$--invariant ample divisor as in the proof of theorem \ref{main}.  
   The $0$--cycle $c_0\cdot  L^4$ is in $A^6_{(0)}(X)$, and so (using lemma \ref{60}) we have
   \begin{equation}\label{here} \iota^\ast (c_0\cdot L^4) = c_0\cdot L^4\ \ \ \hbox{in}\ A^6_{(0)}(X)\ \end{equation}
        
   On the other hand, we have
   \begin{equation}\label{andhere}   \iota^\ast (c_0\cdot L^4)  =  \iota^\ast(c_0)\cdot \iota^\ast(L^4)=  (d_0+ d_2)\cdot L^4 = d_0\cdot L^4 + d_2\cdot L^4 \ \ \ \hbox{in}\ A^6(X) \ .\end{equation}
       
   Since $d_0\cdot L^4\in A^6_{(0)}(X)$ and $d_2\cdot L^4\in A^6_{(2)}(X)$, comparing expressions (\ref{andhere}) and (\ref{here}), we see that we must have
     \[  d_0\cdot L^4 = c_0\cdot L^4\ \ \ \hbox{in}\ A^6_{(0)}(X)\ ,\ \ \ d_2\cdot L^4=0\ \ \ \hbox{in}\ A^6_{(2)}(X)\ .\]
     Using the injectivity part of corollary \ref{hardepw}, this implies that 
      \[   d_2=0\ \ \ \hbox{in}\ A^2(X)\ .\] 
   This proves sublemma \ref{preserve}.
      \end{proof}

   \end{proof} 
 
 Let us now prove corollary \ref{cor}(\rom1). Suppose first $a\in A^6(D)$ is a $0$--cycle in the image of
   \[  A^2(D)\otimes A^2(D)\otimes A^2(D)\ \to\ A^6(D)\ .\]  
   Then $p^\ast(a)\in A^6(X)$ is in the image of
   \[ p^\ast A^2(X)\otimes p^\ast A^2(X)\otimes p^\ast A^2(X)\ \to\ A^6(X)\ .\]
   In view of lemma \ref{in0}, this is contained in the image of
    \[ A^2_{(0)}(X)\otimes  A^2_{(0)}(X)\otimes  A^2_{(0)}(X)\ \to\ A^6(X)\ ,\] 
    which is $A^6_{(0)}(X)$. It follows that $p^\ast(a)$ is rationally trivial if and only if $p^\ast(a)$ has degree $0$. Since $a=2p_\ast p^\ast(a)$, the statement for $a$ follows.
    
  Next, suppose $a\in A^6(D)$ is a $0$--cycle in the image of
    \[ A^3(D)\otimes A^2(D)\otimes A^1(D)\ \to\ A^6(D)\ .\]
    Then $p^\ast(a)\in A^6(X)$ is in the image of
   \[ p^\ast A^3(X)\otimes p^\ast A^2(X)\otimes p^\ast A^1(X)\ \to\ A^6(X)\ .\]
   In view of lemma \ref{in0} and corollary \ref{mck}, this is contained in the image of
   \[  \bigl( A^3_{(0)}(X)\oplus A^3_{(2)}(X)\bigr) \otimes A^2_{(0)}(X)\otimes A^1_{(0)}(X)\ \to\ A^6(X)\ ,\]
   and so we find that
   \[ p^\ast(a)\ \ \in\ A^6_{(0)}(X)\oplus A^6_{(2)}(X)\ .\]
   On the other hand, $p^\ast(a)$ is $\iota$--invariant, and we have
    \[ \Bigl(A^6_{(0)}(X)\oplus A^6_{(2)}(X)\Bigr)\cap A^6(X)^\iota = A^6_{(0)}(X) \]    
    (in view of lemma \ref{60} and theorem \ref{main}). Therefore we must have
   \[  p^\ast(a)\ \ \in\ A^6_{(0)}(X)\ ,\]
   from which the conclusion follows as above.
   
   The proof of corollary \ref{cor}(\rom2) is similar: let $a\in A^5(D)$ be a $1$--cycle in the image of
     \[ A^2(D)\otimes A^2(D)\otimes A^1(D)\ \to\ A^5(D)\ .\]
   Then $p^\ast(a)$ is in the image of
   \[   A^2_{(0)}(X)\otimes  A^2_{(0)}(X)\otimes A^1_{(0)}(X)\ \to\ A^5(X)\ ,\]
   which is contained in $A^5_{(0)}(X)$. But $A^5_{(0)}(X)$ injects into cohomology (this follows from Rie\ss's isomorphism \cite{Rie}, combined with the corresponding statement for
   $A^5_{(0)}(S^{[3]})$ which is noted in \cite[Introduction]{V6}).
   \end{proof}

 The argument proving corollary \ref{cor} actually proves a more general statement:
 
 \begin{corollary}\label{cor2} Let $X$ be a variety of dimension $2m$ of the form
   \[ X= D_1\times\cdots \times D_r\times K_1\times \cdots\times K_s \times X_1\times\cdots\times X_t  \ ,\]
   where each $D_j$ is an EPW cube $D_2^{A_j}$ for $A_j\in\Delta^{1,0}$, and each $K_j$ is a generalized Kummer variety, and each $X_j$ is a Hilbert scheme 
   $(S_j)^{[m_j]}$ where $S_j$ is a $K3$ surface.
   
   Let $E^\ast(X)\subset A^\ast(X)$ be the subring generated by (pullbacks of) 
    \[  A^1(D_j)\ ,\ A^2(D_j)\ ,\ A^1_{(0)}(K_j)\ , \ c_r(K_j)\ ,\ A^1(X_j)\ ,\ c_r(X_j)\ ,\]
    where $c_r()\in A^r()$ denote the Chern classes.
    Then the cycle class map
    \[ E^i(X)\ \to\ H^{2i}(X) \]
    is injective for $i\ge 2m-1$.
   \end{corollary}
   
   \begin{proof} Let us consider the variety
   \[ Y:= Y_1\times\cdots\times Y_r\times K_1\times \cdots\times K_s \times X_1\times\cdots\times X_t  \ ,\]
   where $p_j\colon Y_j\to D_j$ is the double cover from the double EPW cube $Y_j$ to the EPW cube $D_j$, and the finite morphism
   \[ p\colon\ \ \ Y\ \to\ X\ .\]
   The variety $Y$ has an MCK decomposition. (Indeed, the varieties $Y_j$, $K_j$ and $X_j$ have an MCK decomposition, thanks to corollary \ref{mck}, resp. \cite{FTV}, resp. \cite{V6}). As the property of having an MCK decomposition is stable under products \cite[Theorem 8.6]{SV}, the statement for the variety $Y$ follows.) 
   
   There is an inclusion
    \[  p^\ast E^\ast(X)\ \subset\ A^\ast_{(0)}(Y)\ .\]
    (Indeed, we have seen in corollary \ref{cor} that $(p_j)^\ast A^2(D_j)\subset A^2_{(0)}(Y_j)$. Furthermore, it is known that
      \[   c_r(K_j)\ \ \in\ A^r_{(0)}(K_j)\ ,\ \ \  c_r(X_j)\ \ \in\ A^r_{(0)}(X_j)  \ \]
      \cite[Proposition 7.13]{FTV}, resp. \cite[Theorem 2]{V6}.
        Let $\pi$ denote projection from $Y$ to any of the factors $Y_j$ or $K_j$ or $X_j$. Then $\pi$ is ``of pure grade $0$'', in the sense of \cite[Definition 1.1]{SV2}, which means that $\pi^\ast$ preserves the bigrading \cite[Corollary 1.6]{SV2}. This proves the stated inclusion.)
    
    Since 
     \[A^i_{(0)}(Y)\ \to\  H^{2i}(Y)\] 
     is injective for $i\ge 2m-1$, and 
     \[ p^\ast\colon A^i(X)\ \to\ A^i(Y)\] is injective for all $i$, 
     this proves the corollary.   
    \end{proof}

\vskip1cm
\begin{nonumberingt} 
It is a pleasure to thank Len for numerous shared readings of "Het huis van Barbapapa".
\end{nonumberingt}

\end{document}